\documentclass[a4paper,leqno]{mathscan}

\usepackage[all]{xy}
\usepackage{graphicx}
 \usepackage{verbatim}
\usepackage[danish,english]{babel}
\usepackage{color}
\usepackage{enumerate}
\usepackage[abs]{overpic}

\newcommand{\C}{\mathbb C}
\newcommand{\D}{\mathbb D}
\newcommand{\R}{\mathbb R}
\newcommand{\N}{\mathbb N}
\newcommand{\Z}{\mathbb Z}
\newcommand{\imply}{\Longrightarrow}

\newcommand{\fE}{\mathcal E}

\newcommand{\fO}{\mathcal O}
\newcommand{\fP}{\mathcal P}
\newcommand{\union}{\cup}
\newcommand{\intersection}{\cap}

\newcommand{\Kaff}{K_{\text{aff}}}
\newcommand{\Kconf}{K_{\text{conf}}}

\newtheorem{thm}{Theorem}[section]
\newtheorem{pro}[thm]{Proposition}

\theoremstyle{definition}

\newtheorem{defn}[thm]{Definition}  %%%%% here and below, but preferred
\newtheorem{exam}[thm]{Example}     %%%%% by the journal

\usepackage[hidelinks]{hyperref}
\begin{document}

\title[A non FLC regular pentagonal tiling of the plane]
 {A non FLC regular pentagonal tiling of the plane}

\author{Maria Ramirez-Solano}
\thanks{}

\address{Department of Mathematics, University of Copenhagen, Universitetsparken 5,
2100 K\o benhavn \O ,
 Denmark.}

\email{mrs@math.ku.dk}



\thanks{Supported by the Danish National Research Foundation through the Centre
for Symmetry and Deformation (DNRF92), and by the Faculty of Science of the University of Copenhagen.}

\begin{abstract}
In this paper we describe the pentagonal tiling of the plane defined in the article "A regular pentagonal tiling of the plane" by P. L. Bowers and K. Stephenson as a conformal substitution tiling and summarize many of its properties given in the mentioned article.
We show furthermore why such tiling is not  FLC with respect to the set of conformal isomorphisms.
\end{abstract}

\maketitle

\section*{Introduction}
The conformally regular pentagonal tiling of the plane described in the article \cite{StephensonBowers97} is the main character in this work.
The goal is to describe this tiling as a conformal substitution tiling, i.e. a tiling generated by a substitution rule with complex scaling factor $\gamma>1$ and a finite number of prototiles, where  each prototile is substituted with "extended-conformal" copies of the prototiles.
To this end, we begin by stating some results from the theory of Euclidean tilings of the plane. For more details, see \cite{Bellissard06}, \cite{PutnamBible95}, \cite{PutnamCstarKtheory00}, \cite{Sadun08}, \cite{Sol98}.
Then, we introduce the concept of combinatorial tilings. See \cite{TopologyJanich84}, \cite{Hatcher02}, \cite{May99} for more details.
Next we give an overview of the construction of the pentagonal tiling including material from \cite{RiemSurfBeardonBook}, \cite{CirclePackingStephenson} for the overview.
The   main properties of the pentagonal tiling are stated in Proposition \ref{p:propertiesofT}. We should remark that these properties are consequences of the results in  \cite{StephensonBowers97}, and not merely borrowed statements.
We end this work with Theorem \ref{t:TnoFLCwrtConfAutOfC} and some remarks on why we have to abandon the notion of FLC with respect to the set of conformal isomorphisms.

\section{A little theory on tilings of the plane}
A \emph{tile} is a closed subset of the plane homeomorphic to the closed unit disk.
A \emph{tiling} $T$ is a cover of the plane by tiles that intersect only on their boundaries. More precisely, $T$ is a collection of tiles, such that
their union is $\R^2$, and the intersection of any two tiles $t,t'\in T$ is the empty set or a subset of the boundary $\partial t$ of $t$.
There are three ways of constructing interesting tilings and we classify them as (1) substitution tilings, which are described below; (2) cut and project tilings, which invoke a higher dimension and project globally, ie by projection of higher-dimensional structures into spaces with lower dimensionality;  (3) local matching rule tilings, which are jigsaw puzzles of the plane.
These classes are not disjoint, nor they are equal. For example the Penrose tiling is a substitution tiling and a cut and project tiling. See \cite{Sadun08}.\\

Non-periodic tilings give rise to topological dynamical systems, which in turn give rise to $C^*$-algebras which in turn give rise to $K$-groups, which are topological invariants. The construction of a topological dynamical system from a tiling is given right after we remind the reader of the definitions of group action and topological dynamical systems.\\
A \emph{group action} is a triple $(X,G,\phi)$ composed of  a topological space $X$, an Abelian group $G$, and an action map $\phi:X\times G\to X$ defined by $\phi^g:X\to X$, which is a homeomorphism for every $g\in G$, and $\phi^0=id$ and $\phi^g\circ\phi^h=\phi^{g+h}$ for every $g,h\in G$.

A \emph{dynamical system} is a group action $((X,d),G,\phi)$, where $(X,d)$ is a compact metric space called the phase space, and the group action $\phi$ is continuous. For short we write $(X,G)$ instead of $((X,d),G,\phi)$.
The study of the topological properties of  dynamical systems is called  topological dynamics, and
the study of the statistical properties of  dynamical systems is called  ergodic theory. See \cite{Rob91}.

The orbit set of a tiling $T$ is defined by
$$\fO(T):=\{T+x\mid x\in \R^2\},$$
 where $T+x:=\{t+x\mid t\in T\}$. The group $\R^2$ acts on the orbit set $\fO(T)$ of a tiling $T$ by translation, for if $T'$ is in the orbit set, then so is $T'+x$ for all $x\in\R^2$.\\
 The orbit set $\fO(T)$ is equipped with a metric $d:\fO(T)\times \fO(T)\to [0,\infty[$ defined by
 $d(T,T')<\frac 1r$ if there is $x,x'\in B_{1/r}(0)$ such that $(T-x)\cap B_r(0)=(T'-x')\cap B_r(0)$ i.e. if they agree on a ball of radius $r$ centered at the origin up to a small wiggle. See Figure  Figure \ref{f:chairtilingboth}.
 \begin{figure}[htbp]
  \begin{minipage}[b]{0.48\linewidth}
    \centering
    \includegraphics[width=\linewidth]{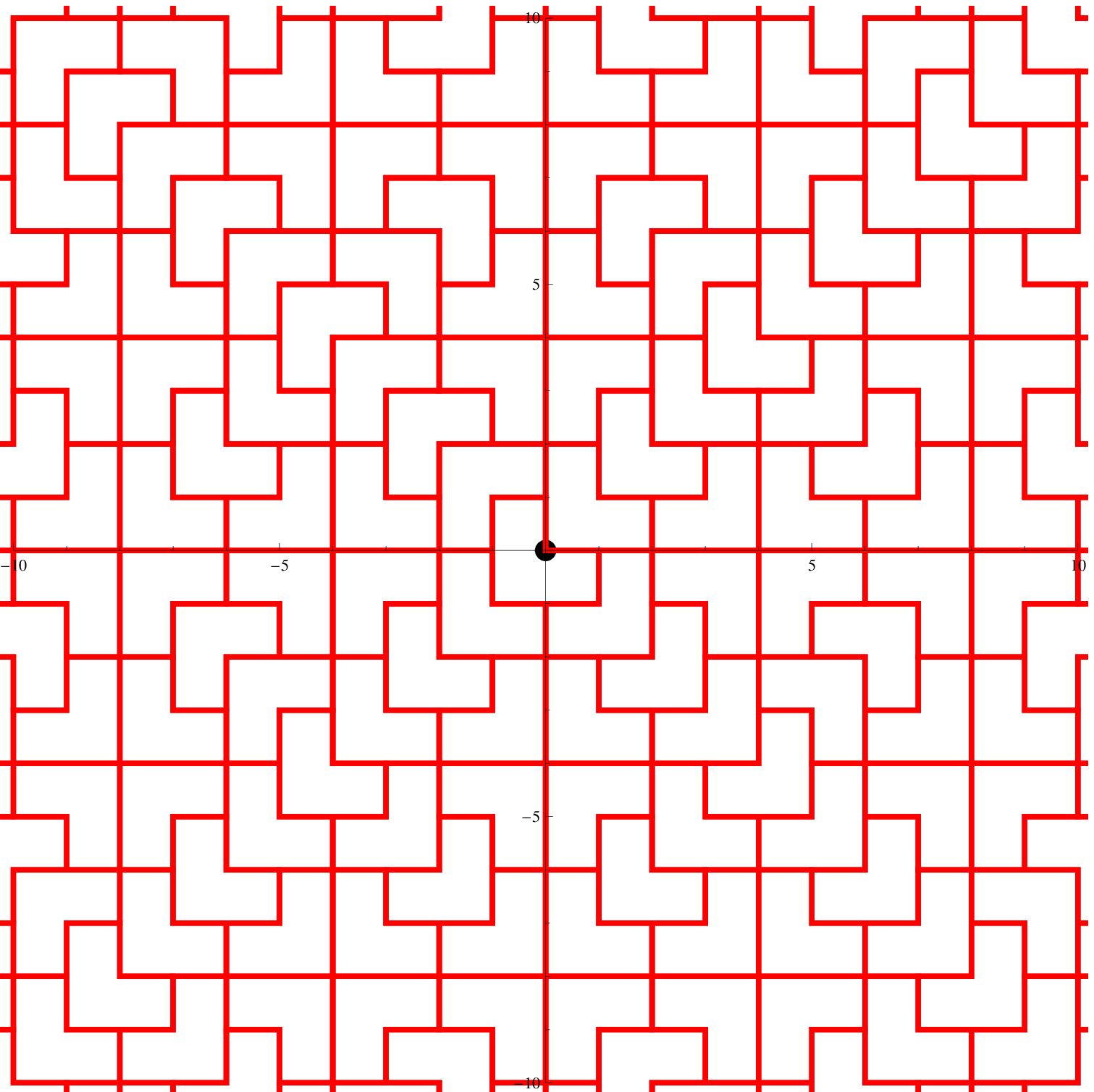}
    \caption{A tiling $T$ (Known as the chair tiling).}
    \label{f:chairtilingT}
  \end{minipage}
  \hspace{0.1cm}
  \begin{minipage}[b]{0.48\linewidth}
    \centering
    \includegraphics[width=\linewidth]{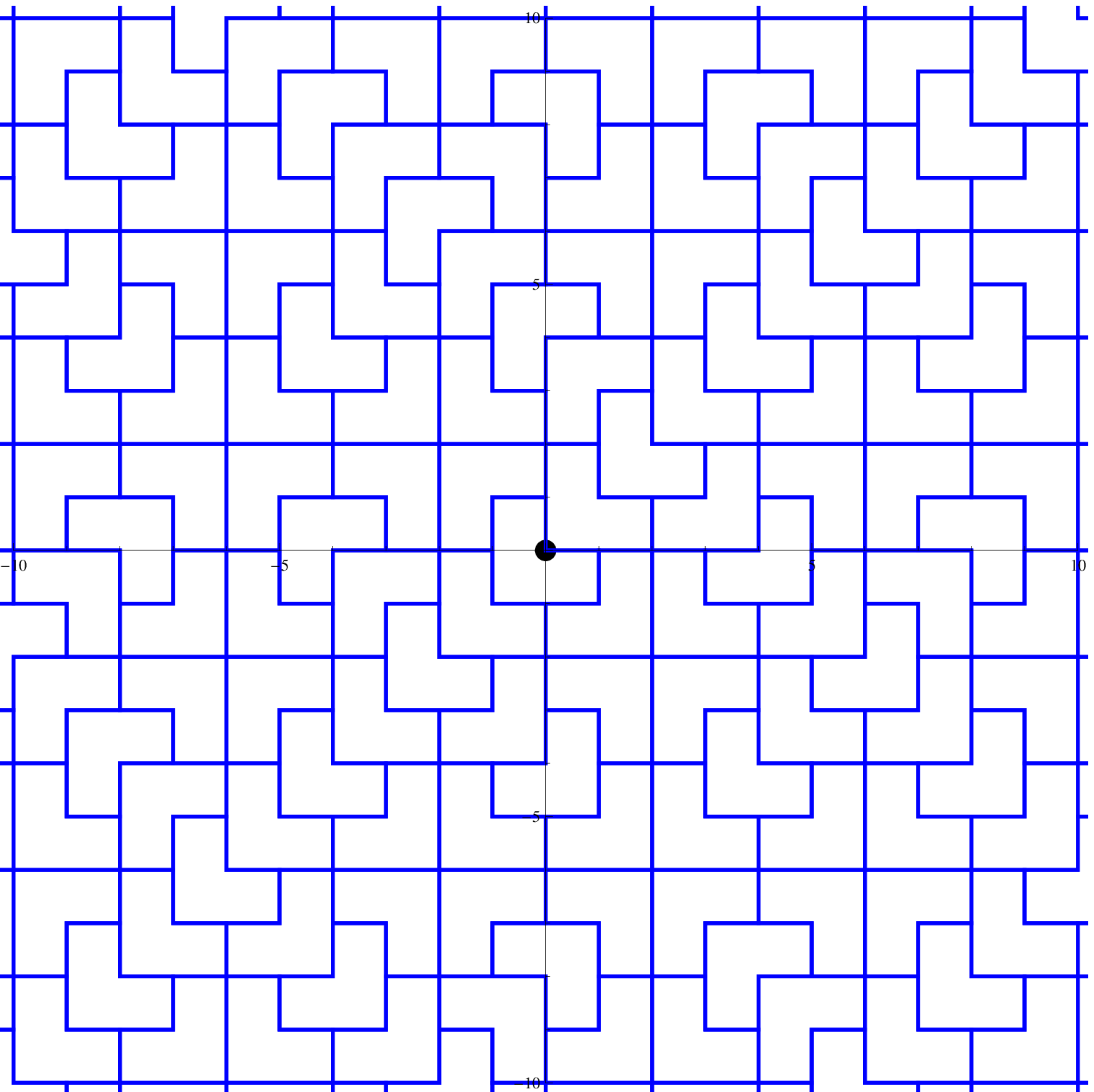}
    \caption{A translate of $T$ given by $T':=T+(2,2)$.}
    \label{f:chairtilingtranslateTp}
  \end{minipage}
\end{figure}

\begin{figure}
  \centering
  \includegraphics[scale=.5]{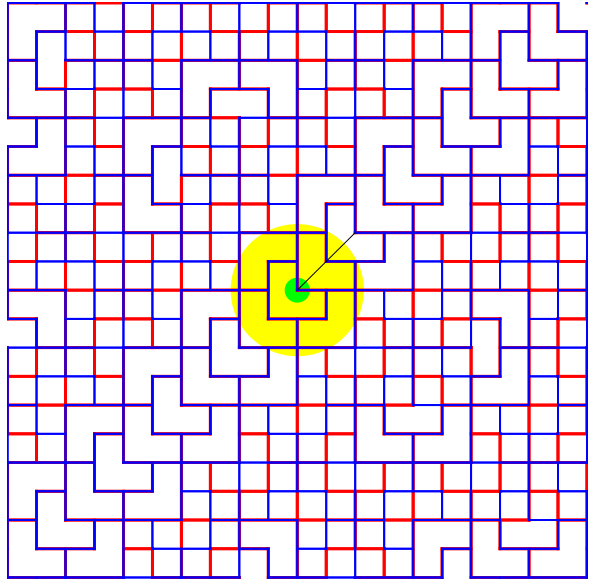}\\\
  \caption{The two tilings $T$, $T':=T+(2,2)$ shown in Figure \ref{f:chairtilingT} and Figure \ref{f:chairtilingtranslateTp} agree on the yellow(large) disk of radius $r$, so the distance $d(T,T')<1/r$. If $T'':=T'+x$ and $x$ is in the  green(small) disk of radius $1/r$, then $d(T,T'')<1/r$.}\label{f:chairtilingboth}
\end{figure}

 The continuous hull $\Omega_T$ of a  tiling $T$ is defined as the completion of the metric space $(\fO(T), d)$, i.e.
  $$\Omega_T:=\overline{\fO(T)}^d.$$
 The same definition of $d$ extends to $\Omega_T$, and  $(\Omega_T,d)$ is a metric space.
The group $\R^2$ acts also on the hull by translation,  for if $T'$ is in $\Omega_T$ then so is $T'+x$ for any $x\in \R^2$. See \cite{PutnamCstarKtheory00}.
A \emph{patch} $P$ is a finite subset of a tiling $T$.
A tiling satisfies the \emph{finite local complexity (FLC)} if for any $r>0$ there are finitely many patches of diameter less than $r$ up to a group of motion $G$, usually translation. The finite local complexity (FLC) is also called \emph{finite pattern condition}.
By Theorem 2.2 in \cite{PutnamCstarKtheory00}, if a tiling $T$ satisfies the FLC condition then the metric space $(\Omega_T,d)$  is compact.
Hence, if a tiling $T$ satisfies the FLC condition, then $(\Omega_T,\R^2)$ is a topological dynamical system. The action $\phi:\Omega_T\times\R^2\to\Omega_T$ given by $\phi^x(T'):=T'+x$ is continuous by definition of the metric.
\\
%(Ian said it).
%FOR DISCRETE STUFF:The action $\phi:\Omega_T\times\R^2\to\Omega_T$ given by $\phi^x(T'):=T'+x$ is continuous because any open subset of $\Omega_T$ remains open when translated by any $x\in \R^2$, and $\R^2$ is equipped with the discrete topology.)

A topological dynamical system $(\Omega_T,\R^d)$ is \emph{minimal} if every orbit is dense.
By construction, the orbit $\fO(T)$ of $T$ is dense in $\Omega_T$.
Hence $(\Omega_T,\R^d)$ is topologically transitive.
A tiling $T$ is said to be \emph{repetitive} if for every patch $P\subset T$ there is an $r>0$ such that every ball of radius $r$ contains a copy of $P$.
A tiling $T$ is aperiodic if $T+x\ne T$ for all $x\ne0$.

By Proposition 2.4 in \cite{PutnamCstarKtheory00}, if a tiling is aperiodic and repetitive then its hull contains no periodic tilings.

For substitution tilings, there is an alternative definition of the hull. See \cite{PutnamBible95}. The construction of a substitution tiling is as follows. Let $P_{prot}:=\{p_1,\ldots,p_N\}$ be a set of tiles of the plane. These tiles are called prototiles.
We define the substitution rule $\omega$ with scaling factor $\lambda>1$ on $P_{prot}$ by
the map $\omega:P_{prot}\to \cup_{i=1}^N \fO(p_i)$ defined by $$p_i\mapsto  \omega(p_i):=\{p_j+x_j\mid p_j\in P_{prot}, x_j\in\R^2\}$$ such that
$\omega(p_i)$ tiles $\lambda p_i$, i.e. the tiles in $\omega(p_i)$ overlap only on their boundaries and their union is $\lambda p_i$. Moreover $\omega(p_i)$ is assumed to be finite.
We extend the definition of $\omega$ to translates of prototiles:
Define $\omega:\cup_{i=1}^N \fO(p_i)\to \cup_{i=1}^N \fO(p_i)$ by
$$\omega(p_i+x):=\omega(p_i)+\lambda x.$$
Observe that $\omega(p_i+x)$ tiles $\lambda(p_i+x)$ and not $\lambda p_i+x$, so  all the points of the set $p_i+x$ are dilated.
Thus, if $(p_i+x)\cap(p_j+y)=e$ then $\lambda(p_i+x)\cap\lambda(p_j+y)=\lambda e$ and so $\omega(p_i+x)\cap\omega(p_j+y)$ tiles $\lambda e$.
A \emph{patch} $P$ is a finite collection of translates of prototiles that overlap only on their boundaries.
Let $\fP_{patches}\subset \fP(\cup_{i=1}^N \fO(p_i))$ be the collection of patches derived from the prototiles.
We now extend the definition of $\omega$ to patches.
Let $\omega: \fP_{patches}\to \fP_{patches}$ be defined by
$$P\mapsto \omega(P):=\bigcup_{t\in P}\omega(t).$$
Observe that $\omega(P)$ tiles the set $\lambda|P|$, where $|P|$ is the union of all the tiles in $P$.
We define the $N\times N$ \emph{substitution matrix} for the substitution map $\omega$ with prototiles $P_{prot}$ by
$A:=(a_{ij})$ where $a_{ij}$ is the number of copies of prototile $p_i$  in $\omega(p_j)$.

It turns out that $\omega^n$ is also a substitution map with scaling factor $\lambda^n$.  Moreover, the substitution matrix of $\omega^n$ is $A^n$.
We say that the substitution map $\omega$ is primitive if its substitution matrix $A$ is primitive.

If the substitution matrix $A$ is \emph{primitive}, then there is $k>1$ such that $(A^k)_{ij}>0$, and hence a copy of each prototile $p_i$ appears inside the patch $\omega^k(p_j)$ for any $i,j$.
If the substitution $\omega$ is primitive, we have $p_j+x\in\omega^k(p_j+y)$ for some $x,y\in\R^2$. We can then construct an increasing sequence of supertiles
$$p_j+x_0\subset \omega^k(p_j+x_1)\subset \omega^{2k}(p_j+x_2)\subset \cdots\omega^{nk}(p_j+x_n)\subset\cdots,$$
for some $x_0,x_1,x_2,\cdots$. Their union might not necessarily cover the entire plane, as $\lambda^{nk}p_j+\lambda^{nk} x_n$ might only cover a section of the plane. The substitution is said to \emph{force its border} if there is a $m\ge1$ independent of $j$ such that $\omega^m(p_j)$ knows its neighbor tiles. Any substitution can be turned into a substitution that forces its border, simply by introducing collared tiles, which are tiles that remember their neighbors and their location relative to the tile. If the substitution forces its border, then the union of the above increasing sequence will cover the plane, and thus we have created a tiling of the plane. If the substitution does not force its border, then we use collared tiles instead together with the new substitution.\\

Define $\Omega$ to be the set of tilings $T$ whose tiles are translations of the prototiles and each patch $P$ in $T$  is contained in a supertile $\omega^n(p_i+x)$ for some $n,i,x$.
If $T$ is a tiling in $\Omega$  then so is $$\omega(T):=\bigcup_{t\in T}\omega(t).$$ We now extend $\omega$ to tilings.
Define $\omega: \Omega \to \Omega$ by $T\mapsto \omega(T)$.

The space $\Omega$ is said to satisfy the \emph{finite local complexity (FLC)} if for each $r>0$ there are only finitely many patches $P\subset T\in\Omega$ of diameter less than $r$ up to translation. Hence if $\Omega$ is FLC then every tiling in it is FLC. Moreover if $T$ is FLC then $\Omega_T$ is FLC.
% putnam said it. see email oct 1 2012.
If the substitution map $\omega$ is primitive and injective, and the space $\Omega$ satisfies the FLC condition then:
 the space $\Omega$  is nonempty by Proposition 2.1 in \cite{PutnamBible95}. The map $\omega$ is surjective by Proposition 2.2 in \cite{PutnamBible95}. The space $\Omega$ contains no periodic tilings by Proposition 2.3 in \cite{PutnamBible95}.\\
 The metric $d$ defined before defines a metric on $\Omega$. Under the same conditions that $\omega$ is primitive and injective, and $\Omega$ satisfies the FLC condition, the group action $((\Omega,d),\R^2)$, where $\R^2$ acts by translation, is a topological dynamical system, and by Corollary 3.5 in \cite{PutnamBible95} it is minimal. Hence for any $T\in\Omega$ we have
 $$\Omega=\overline{\fO(T)}^d=\Omega_T.$$
 Moreover, by Proposition 3.1 in \cite{PutnamBible95}, the substitution map $\omega:(\Omega,d)\to(\Omega,d)$ is a topologically mixing homeomorphism. In particular $\omega$ is continuous.
 Let $T$ be in $\Omega$.
 Let $\Omega'$ be the set of tilings $T'$ such that $T'$ is a tiling whose tiles are translates of the prototiles $P_{prot}$ and $T'$ is locally isomorphic to $T$. By Corollary 3.6 in \cite{PutnamBible95} we have $\Omega=\Omega'$.
 \\
Also, every primitive substitution tiling  is repetitive.
By Theorem 1.1 in \cite{Sol98}, for a  substitution tiling $T$ that satisfies the FLC condition, the substitution map $\omega$ is injective if and only if its hull $\Omega_T$ contains no periodic tilings.
\\
In conclusion we have,\\
\begin{tiny}
\begin{tabular}{ r  l  l }
$T$ FLC $\imply$& $\Omega_T$ compact. &\\
$T$ FLC: $\omega$ injective $\iff$ &$\Omega_T$ has no periodic tilings.&\\
$T$ FLC, substitution tiling, $\omega$ primitive, $\imply$& $T$ repetitive, FLC $\quad\iff$ & $(\Omega_T,\R^d)$ minimal.\\
$T$ substitution tiling, $\omega$ primitive,  $\imply$& $T$ repetitive,aperiodic $\quad\imply$& $\Omega_T$ has no periodic tilings.\\
$\Omega$ FLC, $\omega$ primitive, injective $\imply$ & $\Omega\ne\emptyset$, surjective, has no periodic tilings.&\\
$\Omega$ FLC, $\omega$ primitive, injective $\imply$ & $(\Omega,\R^2)$ minimal. Hence $\Omega=\Omega_T$.&\\.
\end{tabular}
\end{tiny}
\noindent
Thus if $T$ is FLC, aperiodic and $\omega$ is primitive, then $(\Omega_T,\R^2)$ is a minimal topological dynamical system; $\Omega_T$ has no periodic tilings, and $\omega$ is injective.

For each prototile $p_i$ choose a point $x_i\in (p_i)^\circ$ in the interior of it. We say that $x_i$ is a puncture of $p_i$.
The puncture of any translate $t=p_i+x$ is defined by $x(t):=x_i+x$.
Define the discrete hull
$$\Xi:=\{T\in \Omega\mid T\text{ has a puncture at the origin }\}.$$
Since the punctures do not lie on the boundaries, the choice of $p_i-x_i$  is unique.\\
If $T$ is FLC, then $\Xi$ is a Cantor set: compact, totally disconnected, and no isolated points.
A basis for $\Xi$ is given by the cylinder sets $U(P,t)$ defined as follows:
Let $P$ be a patch of $T$ and $t$ a tile in $P$.
Define $U(P,t):=\{T'\in\Omega_T\mid  P-x(t)\subset T\}$ to be the set of tilings that contain $P-x(t)$.
That is, $U(P,t)$ is a subset of $\Xi$ such that the patch $P-x(t)$ is centered at the origin and anything can be outside the patch.
Then $U(P,t)$ is clopen in the relative topology of $\Xi$, and such sets generate the relative topology of $\Xi$. See \cite{PutnamCstarKtheory00}.
\\
The action of translation induces an equivalence relation $R$ on $\Omega$ by declaring two tilings $T,T'\in\Omega$  to be equivalent if they are translates of each other i.e. if $T=T'+x$ for some $x\in\R^2$. Let $R'$ be the restriction of $R$ to $\Xi$, which is an equivalence relation on $\Xi$.
The set  $\Xi$ is a full transversal of $\Omega$ with respect to  $R$ because $[T]_R\cap \Xi\ne\emptyset$ is countable for any tiling $T\in\Omega$.
Observe that the natural map $\Xi\to\Omega/R$ given by $T\mapsto[T]_R$ is surjective but not injective.
\\
By \cite{PutnamCstarKtheory00}, we can construct the $C^*$-algebras $C^*(R')$, $C^*(R)$.
These $C^*$-algebras are strongly Morita equivalent because $\Xi$ is a transversal to $\Omega$ relative to $R$,  and $R'$ is the restriction of $R$ to $\Xi$, and because $\Xi$ satisfies the following three conditions:
%The equivalence relations  $R'$, $R$ are groupoid-equivalent and so their $C^*$-algebras are strongly Morita equivalent.
%They are groupoid-equivalent because there exists a groupoid $G$ such that $[\Xi]_G=[\Omega]_G$ and $R'\cong G|_{\Xi}$ and $R\cong G|_{\Omega}$.
%The topological conditions are also met. These follow because $\Xi$ satisfies the following three conditions:
(1) $T'\in\Omega\imply T'+x\in\Xi$ for some $x\in \R^2$. (2) $T'\in\Xi\imply \{T'+x\mid 0<|x|<\varepsilon\}=\emptyset$ for some $\varepsilon>0$. (3) $\Xi$ is closed in $\Omega$.
\\
The pair $((\Omega,d),\omega)$  is also a dynamical system of its own. The group action is $\Z$, where the homeomorphisms are $\omega^a$, $a\in\Z$. It is a Smale space.

\section{Combinatorial Tilings}
An \emph{$n$-cell} on a topological space $X$ is a subspace homeomorphic to the open $n$-disk $D^n:=\{v\in\R^n\mid ||v||<1\}$. Hence cells are open in $X$.
A \emph{cell decomposition} of a topological space $X$ is a partition of the space into $n$-cells.
\begin{defn}[CW-complex]
A pair $(X,\fE)$ consisting of a Hausdorff space $X$ and a cell decomposition $\fE$ of $X$ is called a \emph{CW-complex} if the following three conditions are satisfied:
\begin{enumerate}
  \item For each $n$-cell $e\in \fE$, there exists a continuous map $\Phi_e:\overline{D}^n\to X$ from the closed unit $n$-disk $\overline{D}^n$ to $X$ that takes the boundary $S^{n-1}$ into the $n-1$ skeleton $X^{n-1}$, and the restriction of this map to the interior $D^n$ is a homeomorphism.
  \item For any cell $e\in \fE$ the closure $\overline e$ intersects only a finite number of other cells.
  \item A subset $A\subset X$ is closed if and only if $A\cap \overline{e}$  is closed in $X$ for each $e\in \fE$.
\end{enumerate}
\end{defn}

A nice result from CW-complexes which will be used extensively is the following:  If $(X,\fE)$ is a CW-complex and $f:X\to Y$ is a map between two topological spaces, then $f$ is continuous if and only if $f$ restricted to each $n$-cell  is continuous if and only if $f$ restricted to each $n$-skeleton $X^n$ is continuous. Moreover, $X^n$ is obtained from $X^{n-1}$ by gluing the $n$-cells in $X$.

The dimension of a CW-complex is said to be $n$ if all the cells are of degree at most $n$.
A \emph{combinatorial tiling} is a 2-dimensional CW-complex  $(X,\fE)$ such that $X$ is homeomorphic to the open unit disk $\D:=D^2$.
The \emph{combinatorial tiles} are the closed 2-cells. We also say that a face is a closed 2-cell, an edge a closed 1-cell, and a vertex a  0-cell.
\begin{exam}
 If $T$ is a tiling of the plane, then $T$ has the structure of a 2-dimensional CW-complex, where the 2-cells are the interior of the tiles, the 1-cells are the interior of the edges of the tiles, and the 0-cells are the vertices of the edges of the tiles. Hence, under this identification, $(\C,T)$ is a combinatorial tiling.
\end{exam}

By the cellular approximation theorem, a continuous map between CW-complexes can be taken to be cellular, i.e. that it maps the $n$-skeleton of the domain to the $n$-skeleton of the range. However, the maps between CW-complexes that we will consider in this work are stronger than cellular maps, for they map cells to cells. We call such maps \emph{cell-preserving} maps.

We borrow the definition of subdivision of a combinatorial tiling from \cite{FloydFiniteSubdivisionRules01}.
\begin{defn}[subdivision of a combinatorial tiling]
  Let $(X,\fE)$ and $(X,\fE')$ be two combinatorial tilings with same topological space $X$.
  We say that $(X,\fE')$ is a subdivision of $(X,\fE)$  if for each cell $e'\in \fE'$,  there is a cell $e\in\fE$ such that $\overline{e'}\subset \overline{e}$.
\end{defn}

\begin{figure}[htbp]
  \begin{minipage}[b]{0.48\linewidth}
    \centering
    \includegraphics[width=\linewidth]{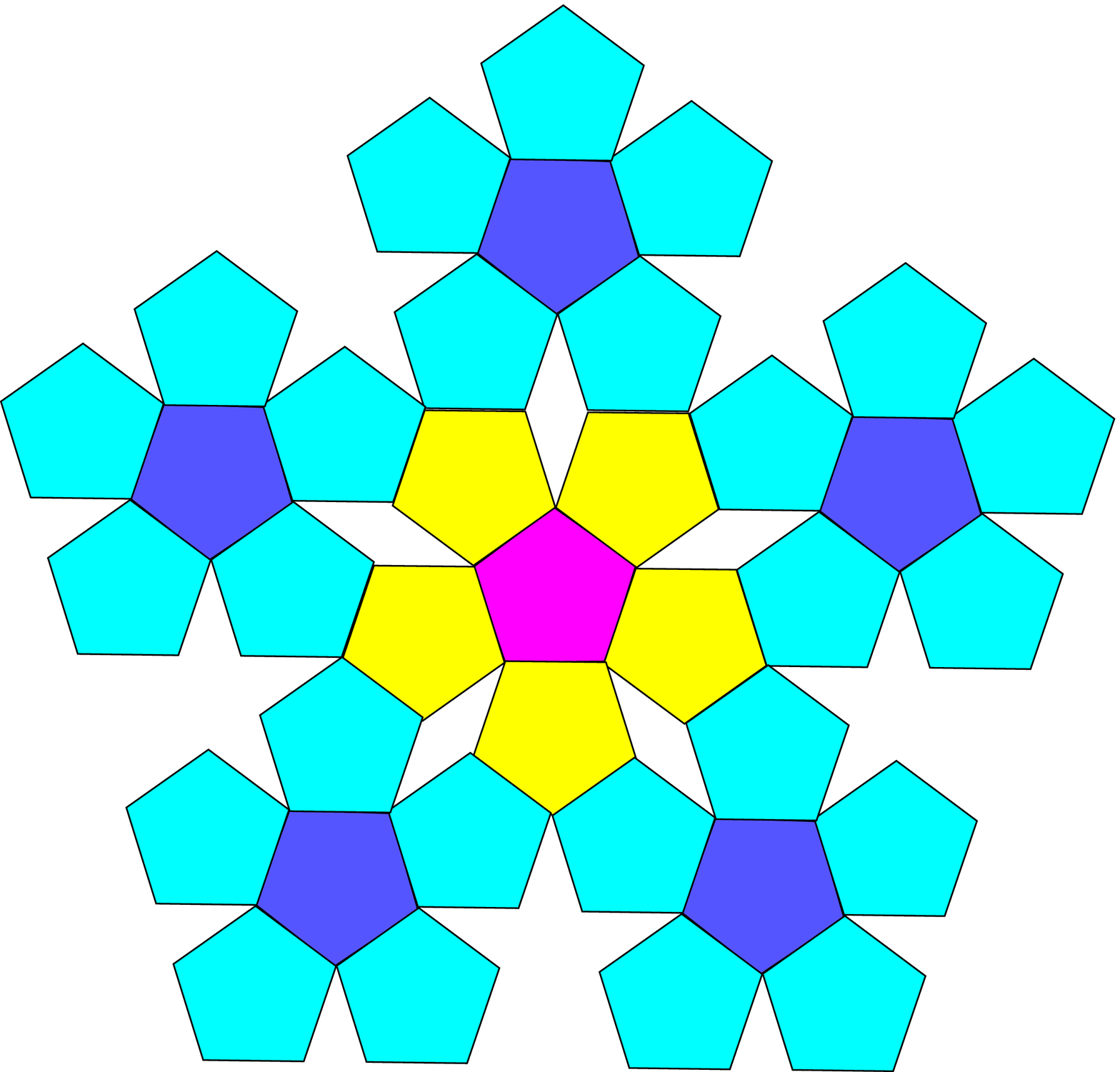}
    \caption{No regular pentagonal tiling on the plane.}
    \label{f:noregularpentagonaltiling}
  \end{minipage}
  \hspace{0.1cm}
  \begin{minipage}[b]{0.48\linewidth}
    \centering
    \includegraphics[width=\linewidth]{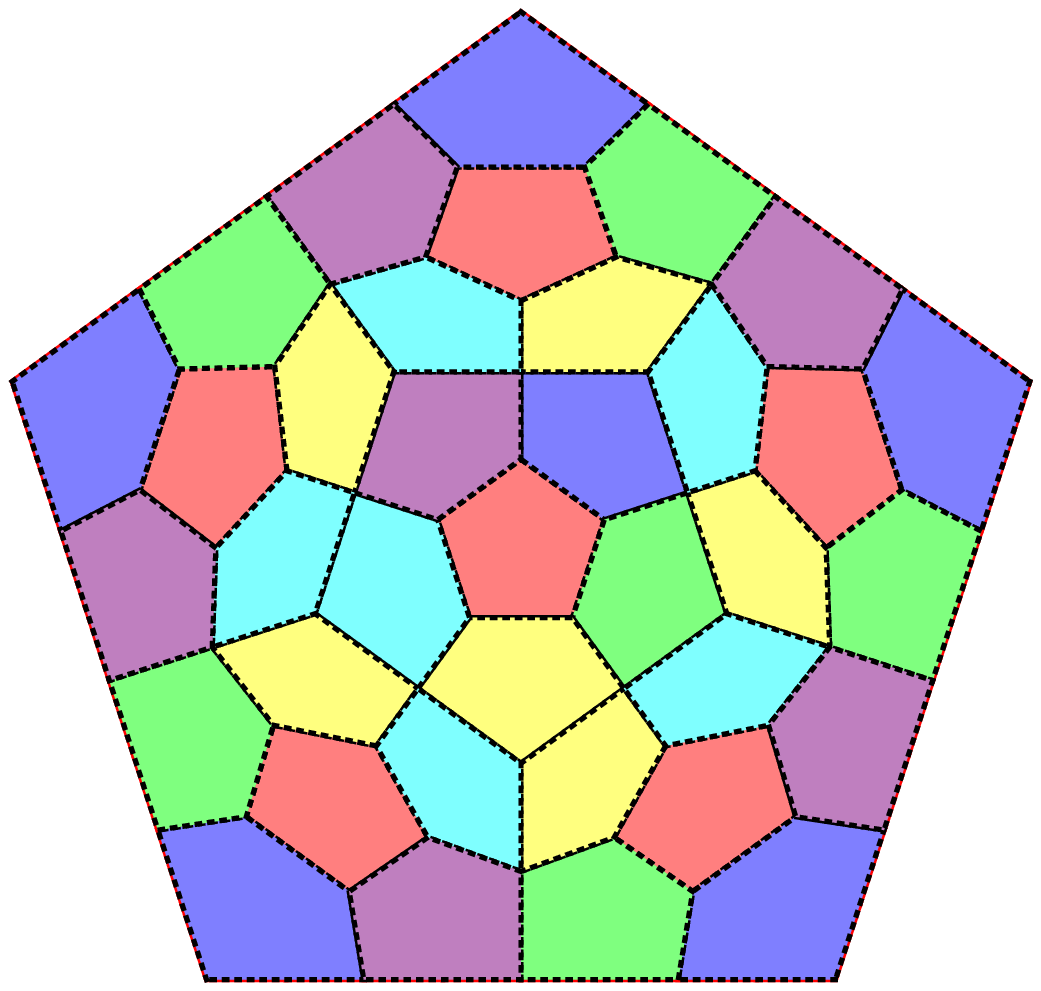}
    \caption{A pentagonal tiling.}
    \label{f:mypentagonaltiling}
  \end{minipage}
\end{figure}

\section{Conformal regular pentagonal tiling of the plane $T$}\label{sectionconstructingconformalT}
Figure \ref{f:noregularpentagonaltiling} shows that we cannot tile the plane with regular pentagons.
However, we could deform the pentagons and obtain a pentagonal tiling  like the one in Figure \ref{f:mypentagonaltiling}.
Although this tiling has many nice properties, we can hope for more. We can construct a tiling with the same combinatorics, but where the tiles are so called conformally regular pentagons, and the tiling looks like the one in Figure \ref{f:conformalregularpentagonaltiling}.
\begin{figure}
  \centering
  \includegraphics[scale=.5]{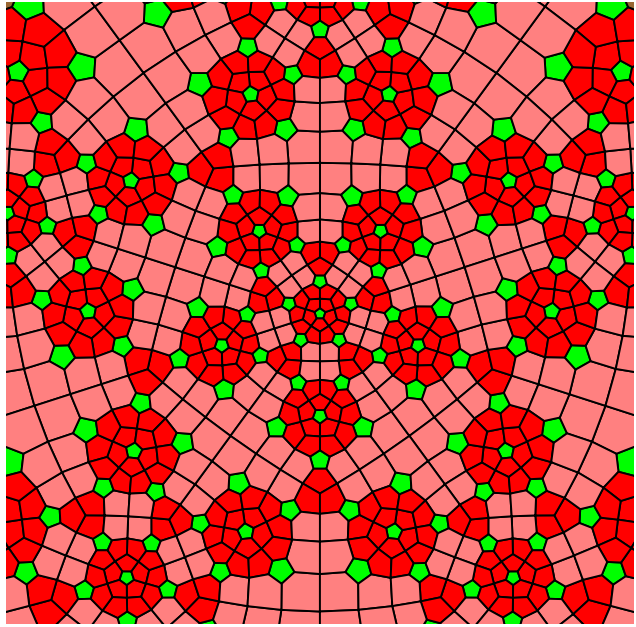}\\
  \caption{The tiling $T$, which is a conformal regular pentagonal tiling of the plane.}\label{f:conformalregularpentagonaltiling}
\end{figure}
The article "A regular pentagonal tiling of the plane" by Philip L. Bowers and Kenneth Stephenson in \cite{StephensonBowers97} gives a construction of this tiling using the theory of circle packings on the above combinatorics. They use circle packings to impose a natural geometry on the above combinatorics.
This gives the nice feature to the tiles of being almost round.
We give a brief summary of the construction and extract/deduce on our own some of the
remarkable combinatorial and geometric properties of the tiling.
Before we proceed, we need to introduce some terminology and results from Riemann surfaces.

\section{Riemann Surfaces}
A \emph{Riemann surface} is a connected Hausdorff topological space $S$ together with an analytic atlas, that is, a Riemann surface is just a one dimensional complex manifold. See pages 33, 39 in \cite{RiemSurfBeardonBook}.
A map $f:S\to S'$ between two Riemann surfaces is said to be analytic if for every chart $\phi_U:U\to\C$ of $S$ and every chart $\phi_V:V\to\C$ of $S'$, the map $\phi_V\circ f\circ \phi_U^{-1}$ is analytic. We should note that a chart $\phi_U$ is a homeomorphism onto its image.
Two Riemann surfaces are said to be \emph{conformally equivalent} if there is a bijective analytic map between them (hence an analytic homeomorphism).
A  connected Hausdorff topological space $S$ with an analytic atlas is always equivalent to the same space $S$  with a different analytic atlas, since the identity map gives a conformally equivalence between them. Hence if we treat conformally equivalent Riemann surfaces as identical, we do not need to worry about which atlas we equip $S$ with.
A \emph{topological closed disk} is a topological space homeomorphic to the closed unit disk. In particular it is compact.
For example a regular Euclidean pentagon is a topological closed disk.
A \emph{conformally regular pentagon } is a Riemann surface $P$ with boundary and five distinguished points on its boundary, which we call corners,  such that: (1) it is a topological closed disk, (2) the interior $P^\circ$ is conformally equivalent to the interior of a regular Euclidean pentagon $P'^\circ$, (3) the bijective analytic map from $P^\circ$ to $P'^\circ$ extends continuously to the boundary, mapping corners to corners.
 We should emphasize that conformally regular pentagons are analytic on their interior, extend continuously to the boundary and they are not conformal on the boundary, so the angles on the boundaries are not preserved.
An example of this is of course a regular Euclidean pentagon. Some more exotic examples of conformally regular pentagons are shown in Figure  \ref{f:conformalprototiles}, and later on we will show why this is the case.
From such figure, we notice that the interior angles on the boundary are not necessarily $2\pi/5$ nor equal, as it should be. However, they are all conformally equivalent to the regular Euclidean pentagon.

Thus we could say that conformally regular pentagons are objects constructed from a regular pentagon.
Two conformally regular pentagons are conformally equivalent on their interiors and homeomorphic on the whole and this homeomorphism preserves the cyclic ordering of their corners.

We will be working with conformally regular pentagons that are subsets of the complex plane. We call such pentagons \emph{conformally regular pentagons of the plane}.
We say that two conformally regular pentagons of the plane are \emph{extended conformally equivalent} if they extend to open sets which are conformally equivalent.
Two pentagons from Figure \ref{f:conformalprototiles} cannot be extended analytically beyond their boundaries since their interior angles do not match. Hence these are not extended conformally equivalent.

\section{Conformal substitution}
We name $T$ the tiling shown in Figure \ref{f:conformalregularpentagonaltiling}, which is also known as a conformal regular pentagonal tiling of the plane. The reason is that all its tiles are extended conformally equivalent to the pentagons shown in Figure \ref{f:conformalprototiles}, and so the boundary angles are preserved.
This tiling is a conformal substitution tiling.
That is, there is a substitution map $\omega$ with complex dilation $\lambda\in\C$ that replaces a tile $t$ with a patch $\omega(t)$ satisfying: (1) the tiles in $\omega(t)$ are conformal regular pentagons which are extended conformally equivalent to the prototiles; (2) the union of all the tiles of $\omega(t)$ is $\lambda t$. The prototiles are the three pentagons shown in Figure \ref{f:conformalprototiles}. The substitution map is shown in Figure \ref{f:conformalomega}, where the complex dilation constant is $\lambda=(-324)^{1/5}\approx 3.2 e^{i \pi/5}$ -  a  dilation by 3.2 and a rotation by $\pi/5$ (page 294 in \cite{CirclePackingStephenson}).
\begin{figure}
  \centering
  \includegraphics[scale=.6]{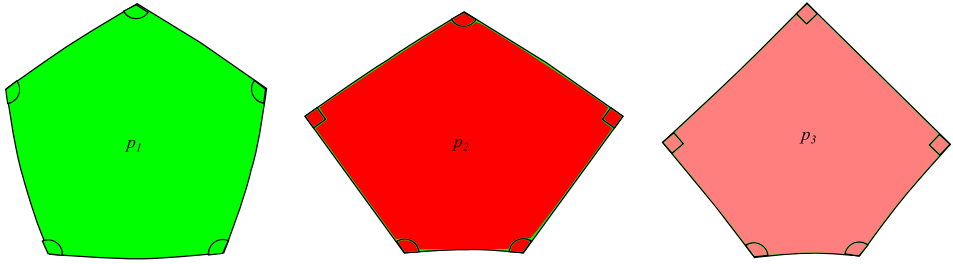}\\
  \caption{The prototiles of $T$. The interior angles are either $\pi/2$ or $2\pi/3$.}\label{f:conformalprototiles}
\end{figure}

\begin{figure}
  \centering
  \includegraphics[scale=.6]{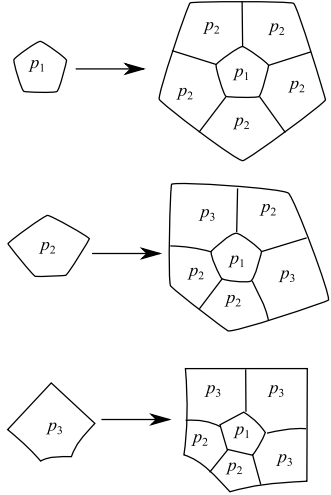}\\
  \caption{The substitution map for the tiling $T$.}\label{f:conformalomega}
\end{figure}
The fact that we are filling $\lambda t$ with conformal copies of the prototiles, and not with traslations/rotations of the prototiles,  is a sign of a generalization of the standard theory for substitution tilings.
The construction in \cite{StephensonBowers97} of the tiling $T$ is as follows.
\section{Construction of the conformal regular pentagonal tiling $T$.}\label{subs:constructionofT}

The story of the construction of $T$  involves five steps. The first one is to construct a combinatorial tiling $K$, which contains the combinatorics of $T$. We then equip it with a piecewise affine structure using regular pentagons, and then with a conformal structure to obtain a simply connected non compact Riemann surface.
Finally, using a generalization of the Riemann mapping theorem, we map this one dimensional complex manifold onto the plane to obtain our tiling $T$.
However, we need circle packing theory to see a drawing of the tiling $T$.
If no confusion arises we will use the same symbol $K$ to denote $K$ with any of the structures, but if we want to emphasize which structure we are referring to then we will write $K$, $\Kaff$, $\Kconf$ to denote that $K$ is equipped with the combinatorial, piecewise-affine, and conformal structure, respectively.
Recall that a combinatorial tiling is a CW-complex homeomorphic to the open unit disk, and so it is a pair consisting of a topological space and a partition.
Usually, no ambiguity arises by denoting the topological space and the partition with the same symbol, for if we talk about cells, we talk about the partition, and if we talk about points, we talk about the space.

The construction of the combinatorial tiling $K$ is the following.
Start with a \emph{combinatorial pentagon} $K_0$, which is a closed topological disk with five distinguished points on its boundary. We write combinatorial to emphasize that at this stage we only care about  combinatorics. That is we can think of it abstractly as one face with five edges and five vertices. Using the combinatorial subdivision rule from Figure \ref{f:subdivisionmap}, we subdivide $K_0$ into six combinatorial pentagons.
The result is a combinatorial flower $K_1$, which is shown in Figure \ref{f:K0K1K2}. We identify $K_0$ with the central pentagon of $K_1$, and we write $K_0\subset K_1$. Repeat this subdivision for the new six pentagons to obtain the combinatorial superflower $K_2$ shown in Figure \ref{f:K0K1K2}.
We identify $K_1$ with the central flower of $K_2$ and so we have $K_0\subset K_1\subset K_2$. Repeating this subdivision $n$ times we obtain an increasing sequence of combinatorial superflowers $K_0\subset K_1\subset \cdots K_n$. The union of all $K_n$ is a combinatorial tiling whose tiles are combinatorial pentagons, each of them attached as in Figure \ref{f:conformalregularpentagonaltiling}.
An alternative construction of $K_n$ from $K_{n-1}$ is using a so called reflection rule, which is shown in Figure \ref{f:reflectionrule}. This rule consists in (1) reflecting the central pentagon $K_0$  across each of its five edges; then (2) the two edges coming out from the central pentagon are identified with each other, and so 5 identifications are made, each corresponding to a corner of the central pentagon, as seen in Figure \ref{f:reflectionrule}. The result is $K_1$. Repeat the same reflection procedure of the superpentagon $K_1$ across each of its five superedges and glue them as before. The result is $K_2$, and so on: Reflecting $K_{n-1}$ across each of its superedges and glue them as in Figure \ref{f:reflectionrule} we get $K_n$.

\begin{figure}
  \centering
  \includegraphics[scale=1]{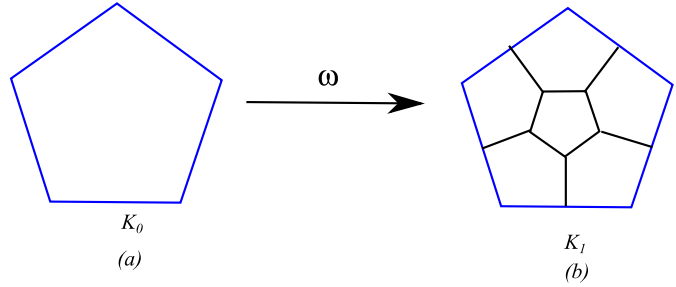}\\
  \caption{Subdivision map. }\label{f:subdivisionmap}
\end{figure}

\begin{figure}
  \centering
  \includegraphics[scale=.5]{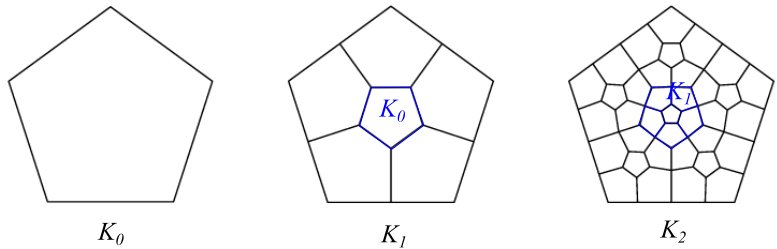}\\
  \caption{The central pentagon $K_0$, combinatorial flower $K_1$, and superflower $K_2$.}\label{f:K0K1K2}
\end{figure}
\begin{figure}
  \centering
  \includegraphics[scale=.7]{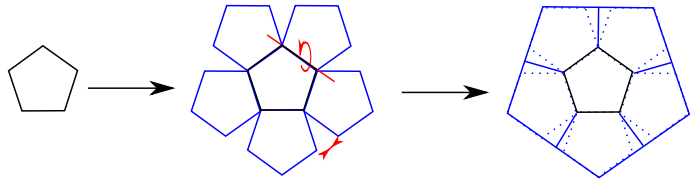}\\
  \caption{Reflection rule.}\label{f:reflectionrule}
\end{figure}

In summary, the combinatorial tiling $K$ is a CW-complex whose cells are shown in Figure \ref{f:conformalregularpentagonaltiling} and whose topological space is (homeomorphic to)  $\C$. It has a central pentagon $K_0$, and its group of combinatorial automorphisms (i.e. cell-preserving homeomorphisms) $Aut(K)$ is the dihedral group $D_5$, which is composed of five rotations with respect to $K_0$, and five reflections with respect to $K_0$ and a vertex of $K_0$. Another important property of the combinatorial tiling $K$ is that  its vertices have either degree $3$ or $4$; that is, each vertex of $K$ is a vertex of either 3 faces or 4 faces.

We now equip $K$ with a piecewise affine structure:
Equip the one skeleton $K^1$ with the unit edge metric, making each edge isometric to the unit interval.
Then extend this metric to faces so that each face is isometric to a regular pentagon of side-length 1.
The distance between two points is defined to be the length of the shortest path between them.
This amounts to replacing each pentagon of $K$ with Euclidean regular pentagons of side-length $1$, all glued along their edges according to the combinatorics of $K$. Observe that on each pentagon we use the Euclidean metric and that the resulting metric ensures compatibility with the combinatorial structure of $K$; for example, the isometric cell-preserving automorphisms of $K$ are still the five rotations and reflections of the dihedral group $D_5$.

As a side note, we wondered how the piecewise affine space $\Kaff$ looked like.
The pentagonal flower $K_1$ is simply a half-dodecahedron. The superflower $K_2$ should look like Figure \ref{f:K2in3D}, but unfortunately, Mathematica shows otherwise. The three degree vertices and flatness of the pentagons, force us to glue a face of 5 dodecahedrons around 5 faces of a dodecahedron, which is not possible - there are gaps in between. The dihedral angle of a dodecahedron is $\arccos \tfrac{-1}{\sqrt5}$, while the angle to be filled is of $\pi -\arccos\tfrac{3}{5}$, which leaves a gap of about 10 degrees.
\begin{figure}[htbp]
  \begin{minipage}[b]{0.48\linewidth}
    \centering
    \includegraphics[width=\linewidth]{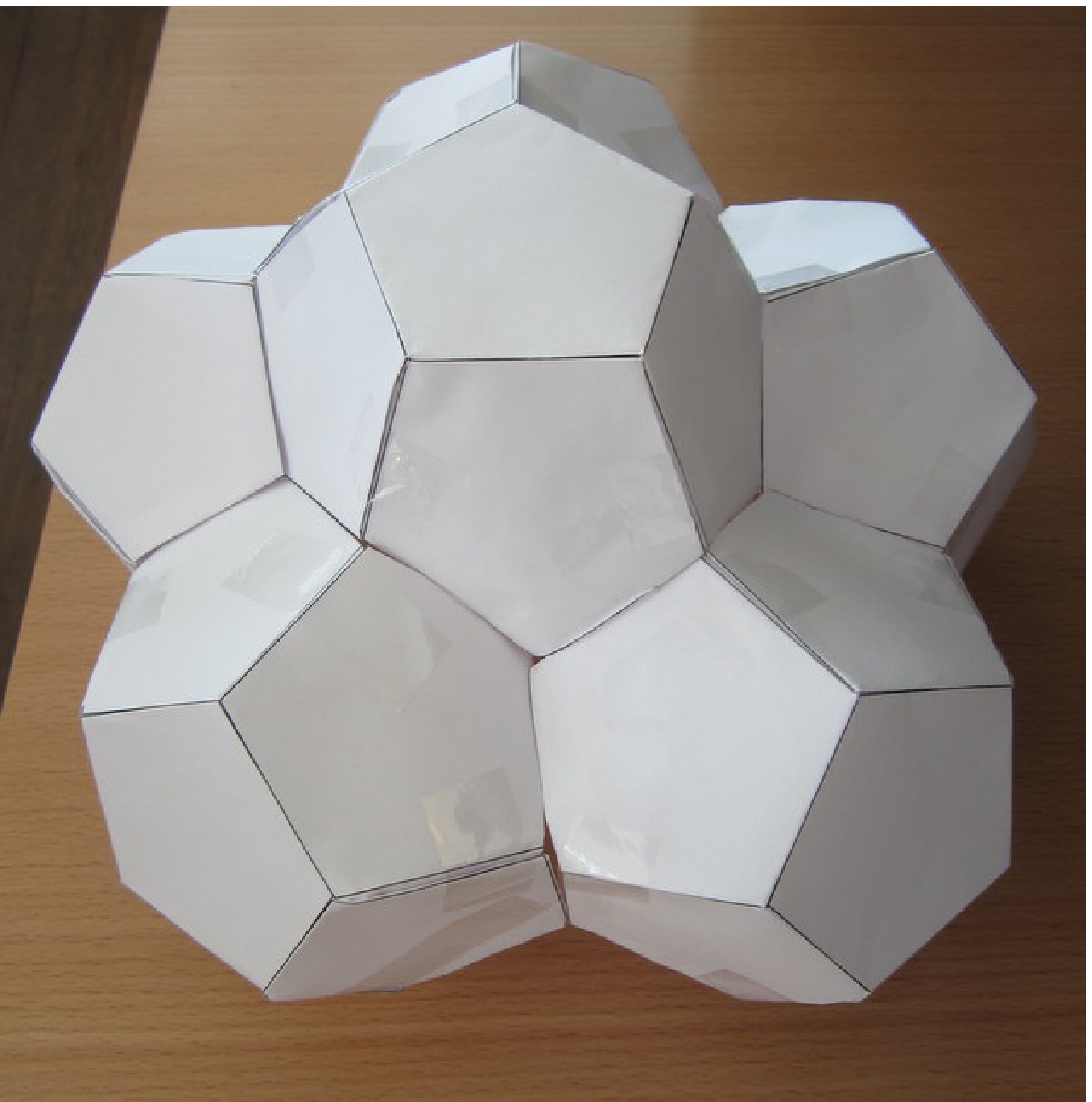}
    \caption{$K_2$ as a gluing of six half dodecahedrons.}
    \label{f:K2in3D}
  \end{minipage}
  \hspace{0.2cm}
  \begin{minipage}[b]{0.48\linewidth}
    \centering
    \includegraphics[width=\linewidth]{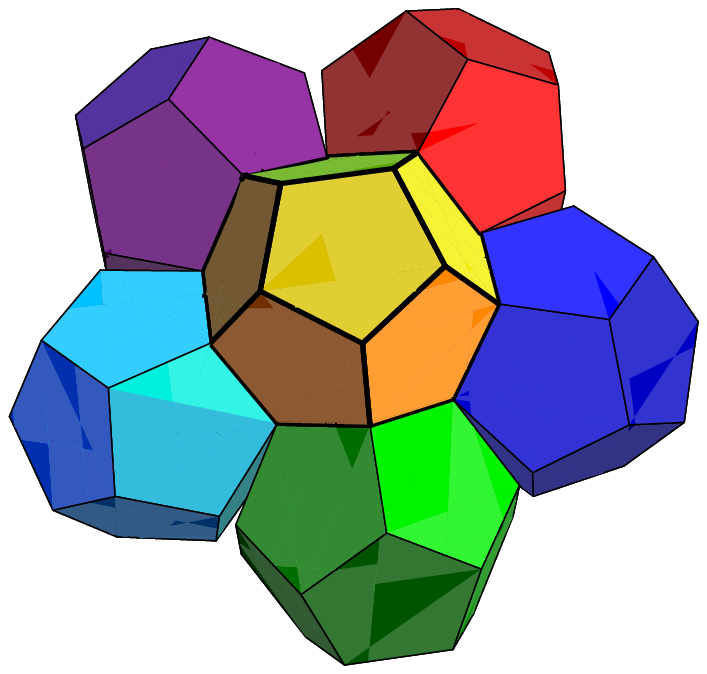}
    \caption{$K_2$ is not the gluing of six half dodecahedrons.}
    \label{f:K2in3Dback}
  \end{minipage}
\end{figure}

We now equip $\Kaff$ with a conformal structure: The piecewise affine space $\Kaff$  is a nonflat surface. The corners and edges will be smoothed by charting them into the plane. For each edge and each vertex we define a chart:
Suppose that $e$ is an oriented edge of $\Kaff$ with pentagon $P_1\in \Kaff$ on its left and pentagon $P_2\in\Kaff$ on its right. See Figure \ref{chartphie}. Define  the open set $U_e:=(P_1\cup P_2)^\circ$ as the interior of the union of the pentagons. We now introduce a coordinate system on $U_e$ by identifying $e$ with the interval $]-1/2,1/2[$ and $P_1$ and $P_2$ with the two regular pentagons $Q_1,Q_2$ that have in common this interval with $Q_1$ being above the $x$-axis and $Q_2$ below. We define $V:=(Q_1\cup Q_2)^\circ\subset \C$ as the interior of the union of these two pentagons; see Figure \ref{chartphie}. The identification map $\phi_e:U_e\to V$ is given by $\phi_e(z):=z$.
The map $\phi_e$ is the chart containing edge $e$, it is indexed by it, and it is an isometry.
The intersection of two domains $U_e$, $U_{e'}$ is (1) is itself if $e$ and $e'$ are the same as nonoriented edges; (2) is a pentagon $P$ if $e$ and $e'$ are two edges of $P$; (3) is the empty set otherwise. For the first case, the transition map is $\phi_e\circ\phi_{e'}^{-1}(z)=-z$ if $e$ and $e'$ have opposite orientation, else $\phi_e\circ\phi_{e'}^{-1}(z)=z$.
For the second case, the transition map is $\phi_e\circ\phi_{e'}^{-1}(z)=(z-z_0) e^{m 2\pi i/5}+z_0 $, where $z_0=\pm\frac{1}{4} i (1+\sqrt{5})$ is the baricenter of the pentagon, $m=0,1,2,3,4$ is the number of times we have to rotate (around the baricenter $z_0$ of the pentagon) the edge $e'$ to match edge $e$, but this is assuming that both edges have same orientation; if they have opposite orientations then $\phi_e\circ\phi_{e'}^{-1}=\phi_e\circ\phi_{-e'}^{-1}(-z)$, where $-e'$ is the edge in the opposite direction of $e'$. See Figure \ref{f:chartphieep}. Hence, the transition maps for charts on edges are all analytic.

\begin{figure}[htbp]
  \begin{minipage}[b]{0.48\linewidth}
    \centering
    \includegraphics[width=\linewidth]{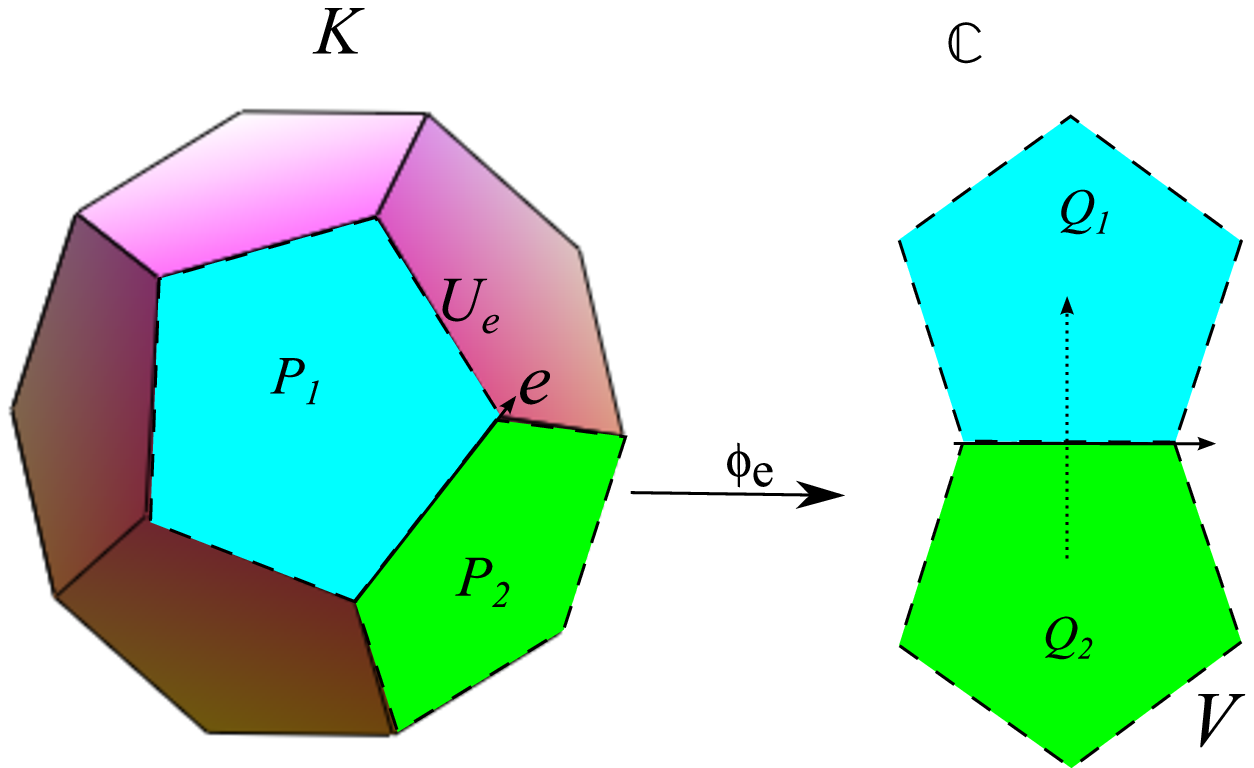}
    \caption{The chart $\phi_e:U_e\to V$, where $U_e=(P_1\union P_2)^\circ$.}
    \label{chartphie}
  \end{minipage}
  \hspace{0.2cm}
  \begin{minipage}[b]{0.48\linewidth}
    \centering
    \includegraphics[width=\linewidth]{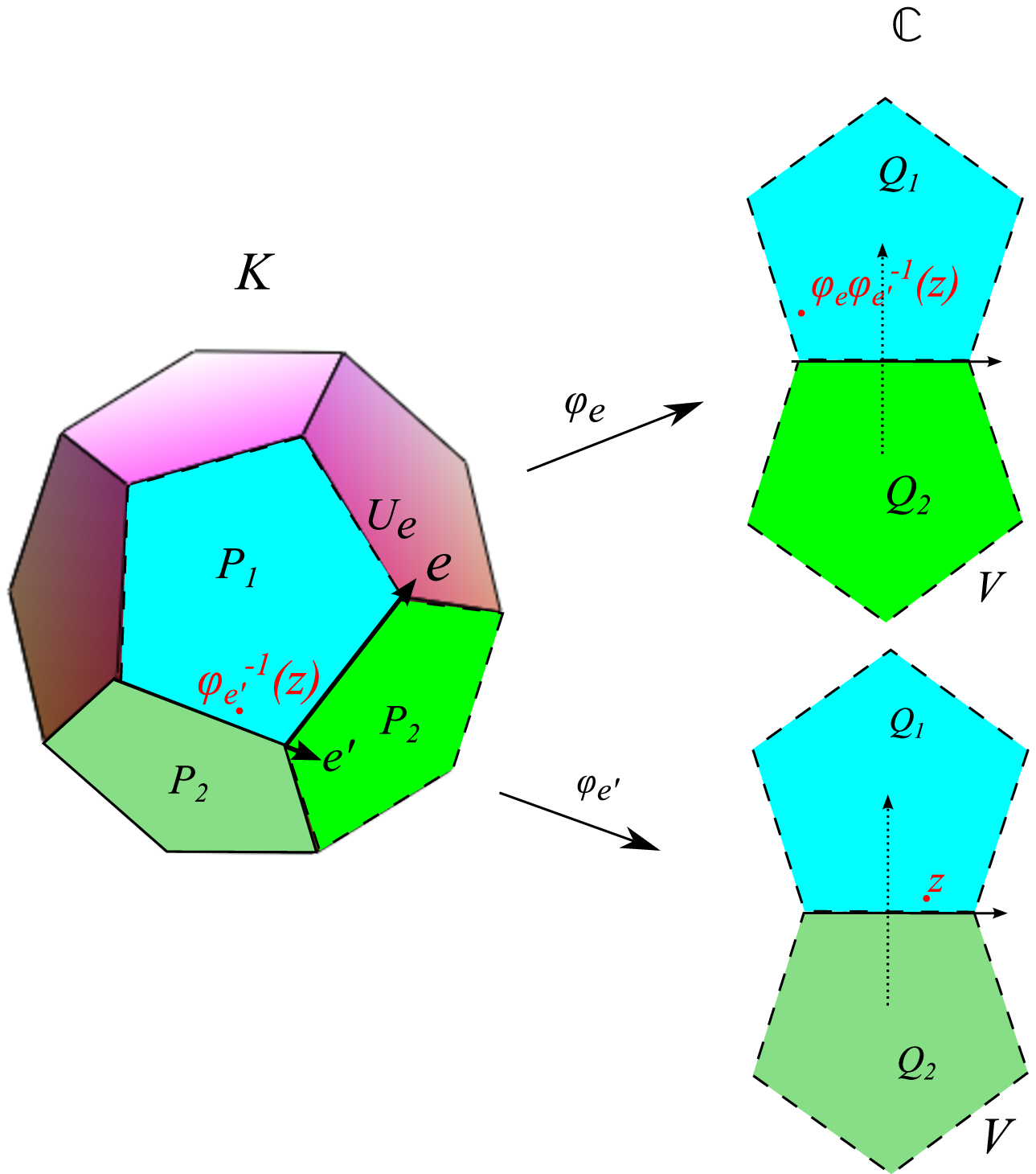}
    \caption{The transition map $\phi_e\phi_{e'}^{-1}(z)=(z-z_0) e^{2\pi i 4/5}+z_0 $, where $z_0=i(1+\sqrt5)/4$.}
    \label{f:chartphieep}
  \end{minipage}
\end{figure}

If $v$ is a $d$-degree vertex of $\Kaff$ and $U_v$ is the open metric ball of radius $1/3$ centered at $v$, and $V'\subset \R^2$ is the open ball of the plane of radius $1/3$ centered at the origin, then the chart $\phi_v:U_v\to V'$ is defined by
$\phi_v(z)=z^{10/(3d)}$, where $z:=r e^{i\theta}$, and $0\le \theta\le \frac{3\pi}{5 }d$ and $0\le r<\frac13$ and the degree of the vertex is $d=3,4$. See Figure \ref{chartphiv} and Figure \ref{chartphivflatten}.
Observe that $\phi_v(r e^{d\frac{3\pi}{5}i})=r^{10/(3d)} e^{d\frac{3\pi}{5}i\frac{10}{3d }}=r^{10/(3d)}e^{2\pi i}=r^{10/(3d)}=\phi_v(r)$, as expected.
Notice that we could as well have used open balls of radius $1/2$ instead of open balls of radius $1/3$ around the vertices.
Observe that the radius of the ball $V'$ is  $(1/3)^{10/(3d)}$ which takes the values $0.295$ and $0.40$.

\begin{figure}[htbp]
  \begin{minipage}[b]{0.48\linewidth}
    \centering
    \includegraphics[width=\linewidth]{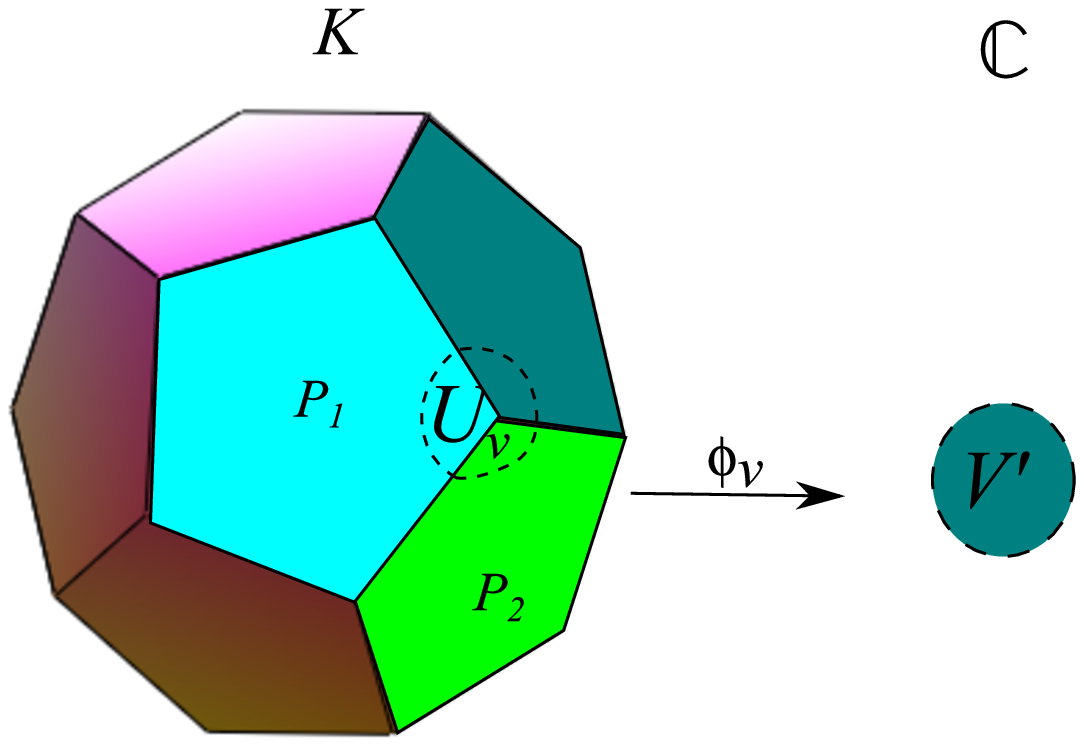}
    \caption{The chart $\phi_v:U_v\to V'$.}
    \label{chartphiv}
  \end{minipage}
  \hspace{0.2cm}
  \begin{minipage}[b]{0.48\linewidth}
    \centering
    \includegraphics[width=\linewidth]{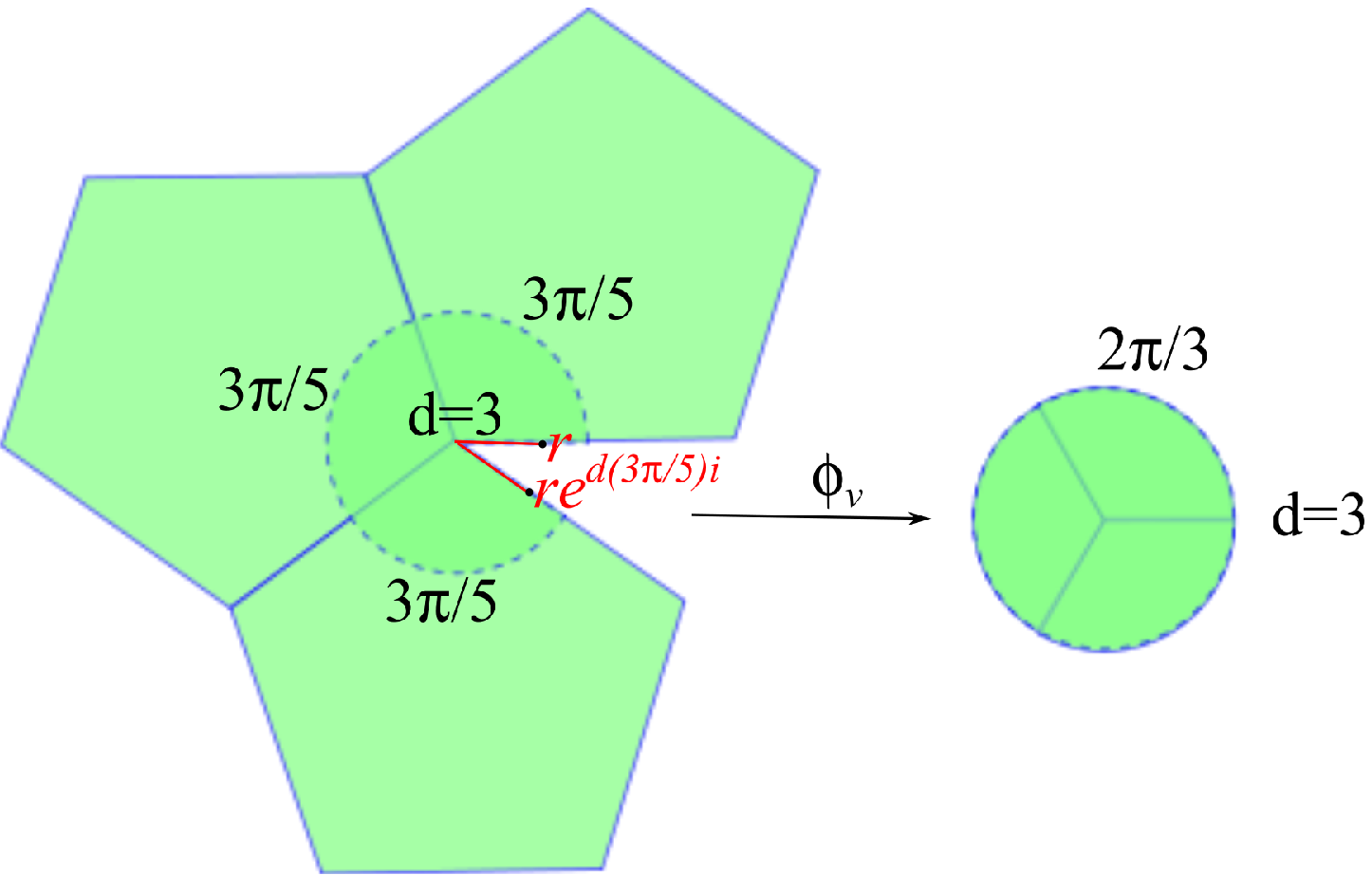}
    \caption{The flatten version of the chart $\phi_v:U_v\to V'$ when $d=3$.}
    \label{chartphivflatten}
  \end{minipage}
\end{figure}

The intersection of the domains $U_v,U_{v'}$ for the charts $\phi_v,\phi_{v'}$ for distinct vertices is always empty.
However, the intersection of the domains $U_e,U_v$ is nonempty whenever $v$ is a vertex of the edge $e$. For such case, the transition maps
$\phi_e\circ \phi_v^{-1}: \phi_v(U_e\intersection U_v)\to \phi_e(U_e\intersection U_v)$,
$\phi_v\circ \phi_e^{-1}: \phi_e(U_e\intersection U_v)\to \phi_v(U_e\intersection U_v)$
are also analytic. For example, if the "$x$-axis" on the ball $U_v$ matches the edge $e$ then $\phi_v\circ \phi_e^{-1}(z)=(z-z_e)^{(10/3d)}$ and $\phi_e\circ \phi_v^{-1}=( \phi_v\circ \phi_e^{-1})^{-1}$, where $z_e:=\phi_e\circ\phi_v^{-1}(0)$. See Figure \ref{f:chartphiev}.
Since the transition maps are analytic maps, we have constructed an atlas - our conformal structure, and thus we have turned $\Kaff$ into a simply connected noncompact Riemann surface $\Kconf$. This structure also preserves the combinatorial structure. For instance, the analytic cell-preserving homeomorphisms on $\Kconf$ are still the five rotations and five reflections of the dihedral group $D_5$.
\begin{figure}
  \centering
  \includegraphics[scale=.7]{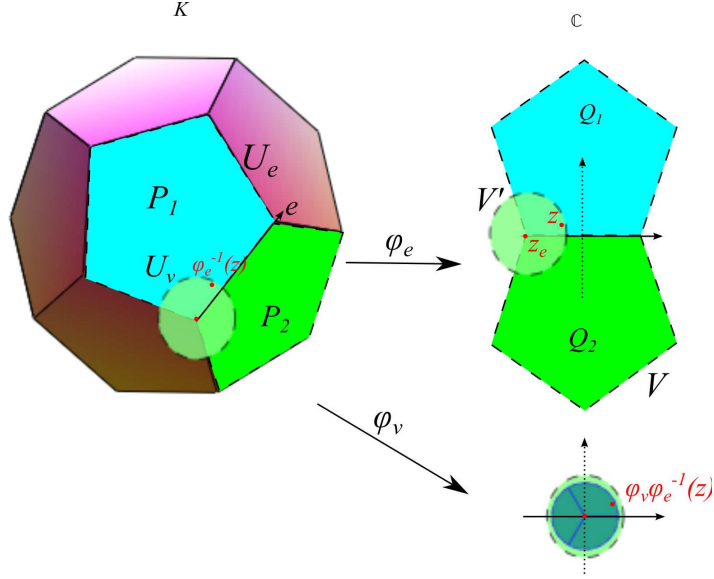}\\
  \caption{A transition map $\phi_v\circ\phi_e^{-1}(z)=(z-z_e)^{10/(3d)}$ with $d=3$.}\label{f:chartphiev}
\end{figure}

The creation of our tiling is obtained as follows.
A generalized version of the Riemann mapping theorem (called the uniformization theorem) tells us that there is an essentially unique analytic map $\Phi$ that sends our Riemann surface bijectively onto either the Euclidean plane or hyperbolic disk.
We define our conformally regular pentagonal tiling of the plane  $T:=\{\Phi(f)\mid f\in K\}$, and so $(\C,T)$ or $(\D,T)$ is a CW-complex, and $\Phi$  becomes a cell-preserving map.
By Schwarz reflection principle, the subdivision map (or more precisely the reflection map) $\omega:K\to K$ induces an analytic cell-preserving homeomorphism $\omega:\Kconf\to \Kconf$, and thus $\alpha:=\Phi\circ\omega\circ\Phi^{-1}$ is a cell-preserving analytic homeomorphism. The fact that $\omega(K)=K$ (i.e. the self similarity of $K$) implies that $\alpha$ increases area.  Since none of the automorphisms on the hyperbolic disk increase area, $\alpha$ must be an automorphism of $\C$. Hence the image of $\Phi$ is the Euclidean plane $\C$.

 The map $\Phi:\Kconf\to\C$ tells us how to flatten our Riemann surface $\Kconf$ onto the Euclidean plane.
Unfortunately, the Riemann mapping theorem (and its generalized version) is an existence theorem. It does not tell us how to draw it.
We use the theory of circle packings to do this, which is explained in full detail in \cite{CirclePackingStephenson}.
First we triangulate our space $\Kaff$ by triangulating each pentagon as follows: Put a vertex on the center of the pentagon and join this vertex with every corner of the pentagon. The resulting triangulation of $\Kaff$ is denoted by $K'$.
The analytic atlas of $\Kconf$ is also an analytic atlas of $K'$ and so $K'$ is a Riemann surface which is conformally equivalent to $\Kconf$.
Circle packing theory gives us an essentially unique collection of Euclidean circles $P_{K'}$ whose centers in $\C$ are indexed by the vertices of this triangulation.
These circles will not overlap, and will follow the combinatorics of $K'$ in the sense that every vertex corresponds to a circle, every edge corresponds to two mutually tangent circles, and every face (i.e. triangle) corresponds to three mutually tangent circles, all respecting the orientation.
The carrier $Carr(P_{K'})\subset \C$ is defined as the CW-complex whose 2-cells are the Euclidean triangles in $\C$  induced by the faces of $K'$.
The carrier is a piecewise affine space. Since it has the same combinatorics of $K'$, we can define a piecewise affine cell-preserving map $\Phi_1:K'\to Carr(P_{K'})$.
The uniform lower bounds on the triangles in $Carr(P_{K'})$ and $K'$ implies that $\Phi_1$ is a $k$-quasi conformal map for some $k>1$. Recall that $k$-quasiconformal (or $k$-qc) maps are generalizations of analytic maps. A $1$-qc map is actually an analytic map, and composition of a $k'$-qc map and a $k''$-qc map is a $k'k''$-qc map, for $k',k''\ge 1$.
Hence, $\Phi_1\circ \Phi^{-1}:\C\to Carr(P_{K'})$ is a $k$-qc map. By Liouville's theorem, there is no quasiconformal mapping of $\C$ onto a proper subset of $\C$, and so $Carr(P_{K'})=\C$.
We now refine the affine space $K'$ into the triangulation $K''$ by introducing a vertex at the midpoint of each edge of $K'$ and then joining these vertices together with new edges. We say that $K''$ is a hexagonal refinement of $K'$.
By the same procedure we obtain a bijective $k_2$-qc map $\Phi_2:K''\to \C$, and by induction we obtain a sequence of bijective $k_n$-qc maps $\Phi_n:K^{(n)}\to \C$.  The maps $\Phi_n\circ\Phi^{-1}$ are also known as discrete conformal mappings.
Since we are doing hexagonal refinement, all $k_n$ are the same as $k$, and therefore  $\Phi_n\circ\Phi^{-1}$ converges to some $k$-qc map $\Psi$, (after some normalization).
By the Rodin-Sullivan Theorem (Theorem 19.1 in \cite{CirclePackingStephenson}, also known as Thurston's conjecture) $\Psi$ restricted to any triangle of $K'$ is analytic and thus, $\Psi$ is analytic everywhere except on a set of measure zero,  which implies that $\Psi$ is analytic everywhere.

Define $T_n:=\{\Phi_n(\Kaff)\mid f\in \Kaff\}$. We say that $T_n$ is an approximation of $T$.
The construction of the first approximation $T_1$ of $T$ (i.e. $n=1$) is shown in Figure \ref{f:approxTstages} and is done as follows:
each $Carr(P_{K_i'})$ is normalized so that their central pentagons agree, where $K_i'$ is the triangulation of $K_i\subset K$; their union converge to the space $Carr(P_{K'})$.
\begin{figure}[htbp]
  \begin{minipage}[b]{0.48\linewidth}
    \centering
    \includegraphics[width=\linewidth]{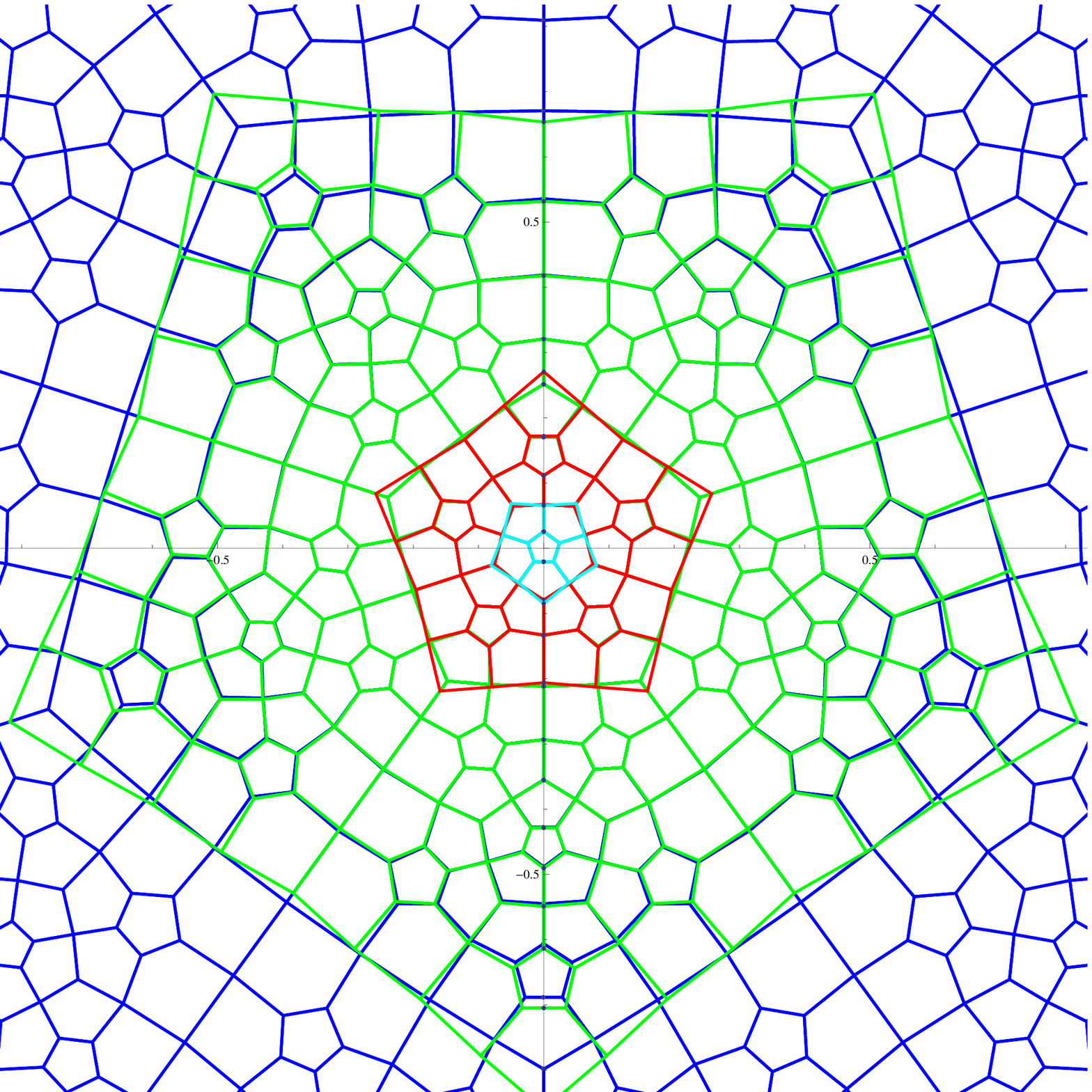}
    \caption{The approximation $T_1$ of $T$.}
    \label{f:approxTstages}
  \end{minipage}
  \hspace{0.2cm}
  \begin{minipage}[b]{0.48\linewidth}
    \centering
    \begin{overpic}[width=\linewidth]{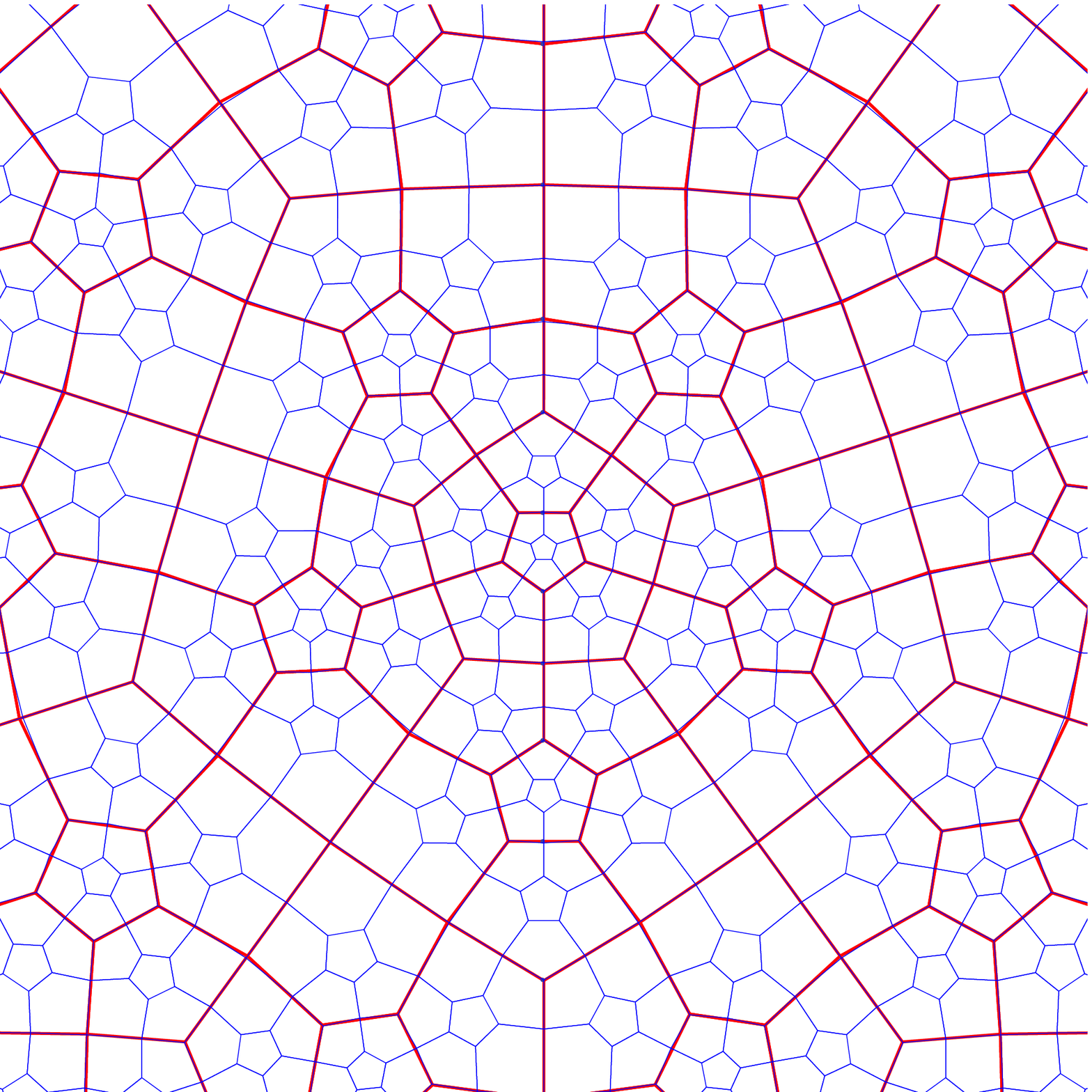}
  \put(85,148){\tiny red}
  \end{overpic}
    \caption{The tilings $T$ (the blue one) and $\lambda T$ (the red one) where $\lambda=|\lambda|e^{(\pi/5) i}$, $|\lambda|=18^{2/5}$.}
    \label{TandaT}
  \end{minipage}
\end{figure}

\noindent
The properties of our tiling $T$ are listed in the following Proposition.
Recall that the automorphism of $\C$ are the complex linear maps.
Two subsets of the plane are said to be \emph{euclideanly similar} if there is a conformal automorphism of $\C$ that maps one subset to the other.
Thus, euclideanly similarity filters out location, scale and rotational effects from the subsets when comparing them.
\begin{pro}\label{p:propertiesofT}
  Our conformal regular pentagonal tiling $T$ has the following properties.
  \begin{itemize}
    \item [(1)] Each  tile of $T$ is extended conformally equivalent to one of the three pentagons shown in Figure \ref{f:conformalprototiles}. We call these three pentagons, the prototiles of $T$.
    \item [(2)] With a single tile $\tau\in T$, we can reconstruct the entire tiling $T$.
    \item [(3)] Each tile $\tau\in T$ is euclideanly similar to at most ten tiles in $T$.
    \item [(4)] The tiling $T$ is aperiodic.
    \item [(5)] The support of the 1-skeleton $(\lambda T)^1$ is a subset of the support of the one-skeleton $T^1$, where $\lambda=(-324)^{1/5}=18^{2/5}e^{i \pi/5}\approx 3.17 e^{i \pi/5}$. See Figure \ref{TandaT}.
    \item [(6)] From the circle packing approximation $T_1$ of $T$ we can obtain good approximations of the tiles of $T$.
    \item [(7)] There are infinitely many tiles of diameter less than $n$ for some $n>1$.
    \item [(8)] The set of diameters of the tiles of $T$ is an unbounded set.
  \end{itemize}
\end{pro}
\begin{figure}
  \centering
  \includegraphics[scale=.7]{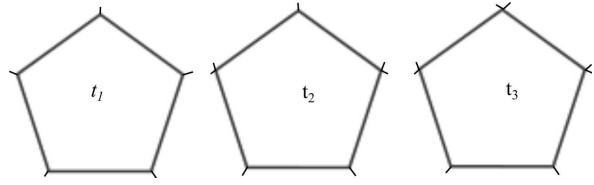}\\\
  \caption{ Prototiles of $K$.}\label{f:prototilesnodec}
\end{figure}
\begin{proof}
  $(1)$.
  There are three prototiles of $\Kconf$, namely  those shown in Figure \ref{f:prototilesnodec}.
  Since $\Phi:(K,atlas)\to\C$ is a conformal map, their images are the prototiles of $T$.
  What remains to be shown are the interior angles of such prototiles.
  If the degree of a vertex $v\in K$ is $d$, then we have $d$ pentagons glued at $v$. The interior angles of these pentagons at $v$ are $3\pi/5$.
  Recall that the angle between two smooth curves $\gamma_1$, $\gamma_2$ which cross at $v$ on a Riemann surface is defined as the angle between the charted curves $\phi_v\circ\gamma_1$, $\phi_v\circ\gamma_2$ at $\phi_v(v)=0$.
  Thus, the interior angles of these pentagons at $v$ charted by $\phi_v$ are  $(3\pi/5)(10/(3d))=(2\pi/d)$.
  Since $d=3,4$, the interior angles of these charted pentagons is either $2\pi/3$ or $2\pi/4=\pi/2$.
  Since $\Phi:(K,atlas)\to\C$ is a conformal map, the interior angles of the pentagons are preserved.

  (2). This is shown on page 9 of the article \cite{StephensonBowers97}. The idea is that each edge $e$ of the tile $\tau$ is an analytic arc, and we can use the corresponding local (anticonformal) reflection to map $\tau$ to the neighboring tile across $e$. By repeated reflections in edges, we reconstruct $T=T(\tau)$.

  (3). Since the automorphisms of the combinatorial tiling $K$ is the dihedral group $D_5$,  the automorphisms of the tiling $T$ are the $\Phi$-conjugations, that is, if $\alpha\in Aut(K)$ then  $\Phi\circ \alpha\circ\Phi^{-1}:\C\to\C$ is a (conformal) automorphism of $T$; in particular, it is a conformal automorphism of $\C$ and  hence a M\"obius transformation. Thus, $T$ has ten automorphisms.
  Since M\"obius transformations of $\C$ are exactly combinations of rotations dilations and translations,  $\tau$ and $\Phi\circ \alpha\circ\Phi^{-1}(\tau)$ are similar.
  If two tiles $\tau_1:=\Phi(f_1)$, $\tau_2:=\Phi(f_2)$ are similar in $\C$, then by (2), the tiling $T(\tau_1)$ is similar to the tiling  $T(\tau_2)$, hence combinatorially equivalent. Thus,
  there exists a combinatorial automorphism of $K$ mapping $f_1$ to $f_2$. By page 9 in the article \cite{StephensonBowers97}, the converse is also true if we use their normalization, which places the center of the tile corresponding to $K_0$ at the origin.
  In summary, two tiles are similar if and only if there is an automorphism of $T$ that maps one tile to the other. Thus, there are at most 10 similar tiles to $\tau$ (five reflections and five rotations with respect to the central pentagon).

  (4) If $T$ was periodic, then there would be an $x\in \R^2$ such that $T$ and $T+x$ are translates, i.e $T=T+x$. Hence $T=T+x=(T+x)+x=T+2x=\ldots=T+nx$, $n\in\N$.
  Since the automorphisms of $T$ are in total ten, a finite number, and a translation is an automorphism, $T$ cannot be periodic.

  (5),(6). In \cite{StephensonBowers97}, it is shown that the central tile of $T$ can be subdivided with the subdivision rule in a conformal way, and using Schwarz reflection principle, we can divide the whole $T$ . The result is a tiling $T'$ whose combinatorics are $\omega(K)$, and an analytic bijective map $\alpha:T\to T'$ mapping each tile to its subdivision. Hence  $\omega=\Phi^{-1}\circ\alpha\circ\Phi$, and thus $\omega$ is analytic.  In page 9 of the same article, it is shown that the map $\alpha:\C\to\C$ is given by $\alpha(z):=\lambda z$, $\lambda=(-324)^{1/5}$ (under some normalization settings).  Hence (5).
  We show (6) only when the tile is the central pentagon, since we have notation for this, and the same argument applies for any other tile of $T$.
  Let $\tau_n:=\Phi(K_n)$. Since $K_n$ is the subdivision of $K_{n-1}$, $\tau_n=\alpha(\tau_{n-1})=\lambda \tau_{n-1}=\cdots=\lambda^n \tau_{0}$.
  Hence $\tau_n$ and $\tau_0$ are euclideanly similar.
  Recall that $T_1:=\Phi_1(K)$ is the first approximation of $T$.
  Let $t_n:=\Phi_1(K_n)$ be the first approximation of the supertile $\tau_n$.
  Since the interior angles of the boundary of $t_n$ are closer to $2\pi/3$ the bigger $n$ gets, $t_n$ gives a better approximation of $\tau_0$ than $t_0$. See Figure \ref{f:approxt0}.

  (7) In the proof of  Proposition 5.1 of the article, it is shown that for each tile $\tau$ there is a sequence $(\tau_j^n)_{n\in\N}$ of tiles that converge in Hausdorff metric to $b\bar\tau$, where $b\approx 1.3$. See Figure \ref{f:seqforconstantb} and Figure \ref{f:hausdorfflimitconstantb}. Recall that the Hausdorff metric is defined on compact sets $A,B$ by $d(A,B)=\max(\max_{x\in A}d'(x,B),\max_{y\in B}d'(A,y))$, where $d'(x,B)$ is the smallest distance between $x$ and $B$. See Figure \ref{hausdorffmetric}.
  \begin{figure}[htbp]
  \begin{minipage}[b]{0.48\linewidth}
    \centering
    \includegraphics[width=\linewidth]{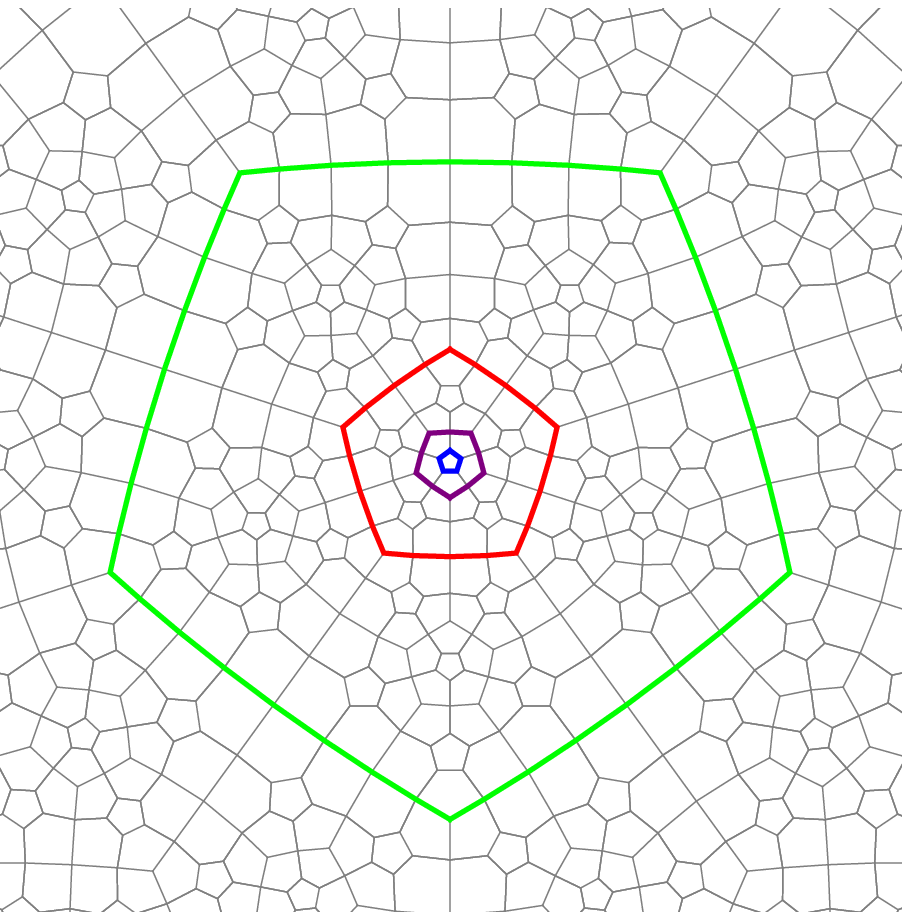}
    \caption{The approximations $t_i$, $i=0,1,2,3$ of the central tile $\tau_0$.}
    \label{f:approxt0}
  \end{minipage}
  \hspace{0.2cm}
  \begin{minipage}[b]{0.48\linewidth}
    \centering
    \includegraphics[width=\linewidth]{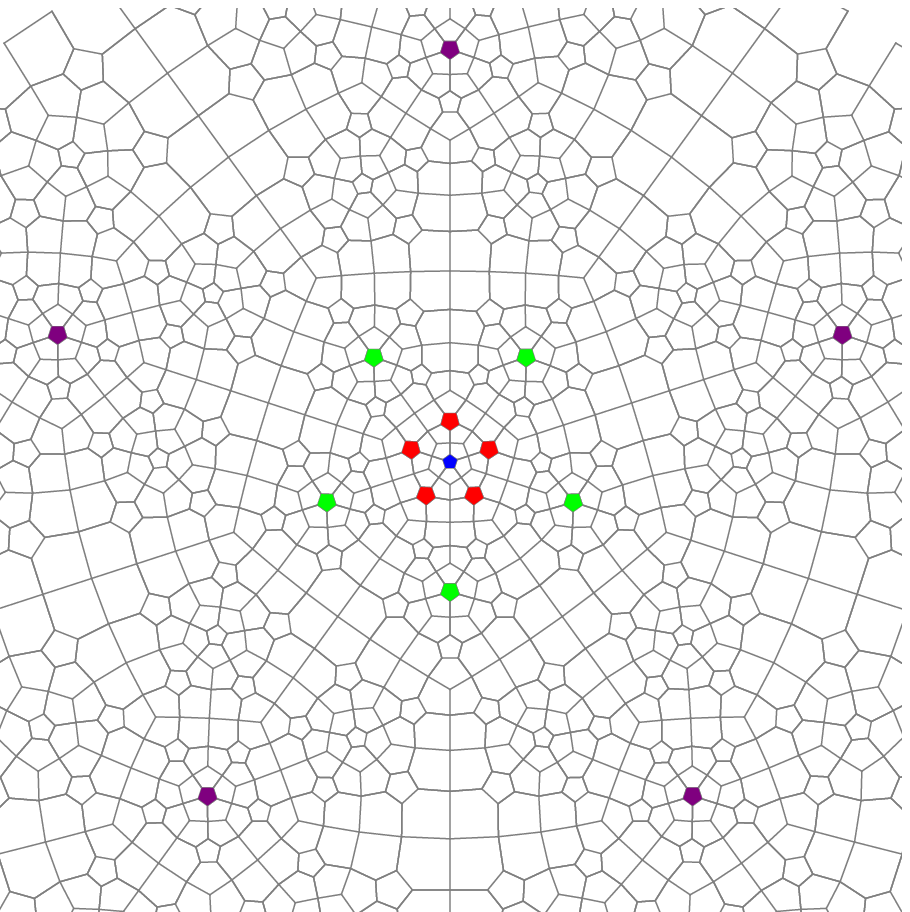}
    \caption{The colored pentagons, once normalized as in Figure \ref{f:hausdorfflimitconstantb}, converge  in Hausdorff measure  to $b\,\overline{\tau_0}$, where $\tau_0$ is the central pentagon.}
    \label{f:seqforconstantb}
  \end{minipage}
\end{figure}
\begin{figure}
  \centering
  \includegraphics[scale=.2]{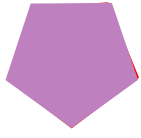}\\\
  \caption{The colored pentagons (rotated and centered) from Figure \ref{f:seqforconstantb} converge  in Hausdorff measure to  $b\,\overline{\tau_0}$.}\label{f:hausdorfflimitconstantb}
\end{figure}

\begin{figure}
  \centering
  \includegraphics[scale=.5]{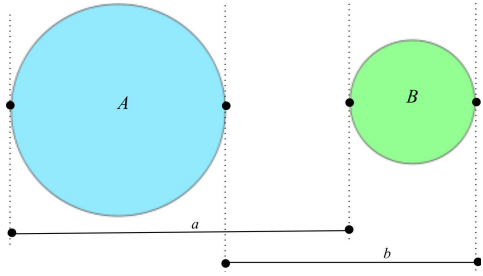}\\\
  \caption{Define $a:=\max_{x\in A}d'(x,B)$,  $b:=\max_{y\in B}d'(A,y))$. The Hausdorff distance between $A$ and $B$ is $max(a,b)=a$.}\label{hausdorffmetric}
\end{figure}

  (8) Let $\tau_0\in T$ be a tile. By the proof of (7), there is a tile $\tau_{1}$ satisfying
   $b'diam(\tau_0)<  diam(\tau_{1})$, where $b'<b\approx 1.3$, say $b':=1.1$.
  Repeating the same argument for $\tau_1$, we find a tile $\tau_2$ satisfying  $b'^2 diam(\tau_0)<  b' diam(\tau_1)<  diam(\tau_{2})$.
  Repeating the same argument $n$ times we obtain a tile $\tau_n\in T$ satisfying $b'^n diam(\tau_0)< diam(\tau_n)$ of any size.
\end{proof}

%In conclusion, the tiling $T$ is very rigid and investigating $T$ is basically the same as investigating the combinatorial tiling $K$.\\

\begin{thm}\label{t:TnoFLCwrtConfAutOfC}
  The tiling $T$ does not satisfy the finite local complexity with respect to the group of isometries.
\end{thm}
\begin{proof}
 Let $G$ be the group of isometries.
  Suppose that $T$ has FLC with respect to $G$.
  Let $\tau_0$ be the central tile of $T$.
  By part (7) of the previous Proposition \ref{p:propertiesofT}, there is a sequence of tiles of diameter less than $b\cdot diam(\tau_0)$, $b\approx 1.3$.
  Since $T$ has FLC with respect to $G$,  there is a finite number of patterns; that is, each of these tiles are euclideanly similar to a finite number of them. This is a contradiction, because by part (3) of the previous Proposition \ref{p:propertiesofT} each tile is euclideanly similar to at most 10 other tiles.
\end{proof}

 Informally, FLC says that if we consider all regions of a fixed size, up to some notion of equivalence, there are only finitely many. We don't need a group; we need notion of equivalence. In the classical theory, everything is sitting in the plane and it is natural to use translations or isometries.
 If the notion of equivalence we are using does not preserve the notion of size, it is useless, for what is the point of taking two regions of size $r$ and asking if there is an isomorphism between them that doubles the size? This isn't the end of the problems because we will lose compactness and so on as well. But it is enough to give up on the idea. Take for example the sequence of tiles of diameter less than $b\cdot diam(\tau_0)$, $b\approx 1.3$ from the proof of Theorem \ref{t:TnoFLCwrtConfAutOfC} with the additional property that they shrink by the factor $1.3$, but the patches around them are converging combinatorially to $K$.  What does this sequence converge to? We don't know.

 Thus we have to abandon the notion of FLC with respect to the set of conformal isomorphisms that are defined between open subsets of the plane.
  However, we will show in another article that the combinatorial tiling  $K$ has FLC with respect to the set of isomorphisms that are defined between subcomplexes of $K$.

\section*{Acknowledgments}
The results of this paper were obtained during my Ph.D. studies at University of Copenhagen. I would like to express deep gratitude to my supervisor Erik Christensen and to Ian F. Putnam, Kenneth Stephenson, Philip L. Bowers, Bill Floyd, whose guidance and support were crucial for the successful completion of this work.

% -----------------------------------------------------------
\bibliographystyle{amsplain}
%\bibliography{xbib}

\end{document}